\documentclass[11pt]{article}
\usepackage[T1]{fontenc}
\usepackage{tgtermes}
\usepackage[square,sort,comma,numbers]{natbib}
\usepackage{graphicx}
\usepackage{amssymb}
\usepackage{amsmath}
\usepackage{latexsym}
\usepackage[left=1.2in,top=1.2in,right=1.2in,bottom=1.2in]{geometry}\usepackage{amsmath}
\usepackage{booktabs}
\usepackage{color}
\usepackage{tikz}
\usepackage{booktabs}
\usepackage{caption}
\usepackage{float}
\usepackage{titlesec}
\usepackage{capt-of}
\usepackage{array}
\usepackage{arydshln}
\usepackage{graphicx}
\usepackage{subcaption}
\usepackage[section]{placeins}
\usepackage{authblk}
\usepackage{amsthm}

\usetikzlibrary{shapes}
\usetikzlibrary{arrows}

\definecolor{darkblue}{rgb}{0,0.4,0.9}
\definecolor{gray10}{rgb}{0.1,0.1,0.1}
\definecolor{gray20}{rgb}{0.2,0.2,0.2}
\definecolor{gray30}{rgb}{0.3,0.3,0.3}
\definecolor{gray40}{rgb}{0.4,0.4,0.4}
\definecolor{gray60}{rgb}{0.6,0.6,0.6}
\definecolor{gray80}{rgb}{0.8,0.8,0.8}
\definecolor{gray90}{rgb}{0.9,0.9,.9}
\definecolor{gray95}{rgb}{0.95,0.95,.95}
\definecolor{gray96}{rgb}{0.96,0.96,.96}
\definecolor{lgreen} {RGB}{180,210,100}
\definecolor{dblue}  {RGB}{20,66,129}
\definecolor{ddblue} {RGB}{11,36,69}
\definecolor{lred}   {RGB}{220,0,0}
\definecolor{nred}   {RGB}{224,0,0}
\definecolor{norange}{RGB}{230,120,20}
\definecolor{nyellow}{RGB}{255,221,0}
\definecolor{ngreen} {RGB}{98,158,31}
\definecolor{dgreen} {RGB}{78,138,21}
\definecolor{nblue}  {RGB}{28,130,185}
\definecolor{jblue}  {RGB}{20,50,100}
\definecolor{nnyellow}{RGB}{235,200,0}
\definecolor{purple}{RGB}{150, 0, 120}
\definecolor{sgGreen} {RGB}{20, 180, 50}
\definecolor{revised}{rgb}{0,0,0.9}

\newtheorem{theorem}{Theorem}

\newtheorem{lemma}{Lemma}

\newcommand{\pl}{\parallel}
\newcommand{\openr}{\hbox{${\rm I\kern-.2em R}$}}
\newcommand{\openn}{\hbox{${\rm I\kern-.2em N}$}}
\newcommand{\opend}{\hbox{${\rm I\kern-.2em D}$}}

\title{Robust Estimation of Data-Dependent Causal Effects based on Observing a Single Time-Series}

\author[1]{Mark van der Laan} 
\author[2]{Ivana Malenica}
\affil[1]{\small{Department of Statistics and Division of Biostatistics, University of California, Berkeley}} 
\affil[2]{\small{Division of Biostatistics, University of California, Berkeley}} 
 
\date{\today}

\begin{document}
\maketitle
 
\begin{abstract}
Consider the case that one observes a single time-series, where at each time $t$ one observes a data record $O(t)$ involving treatment nodes $A(t)$, possible covariates $L(t)$ and an outcome node $Y(t)$.
We assume that the conditional distribution of $O(t)$, given the observed past, is described by a common function only depending on a fixed dimensional summary measure of the past ($C_o(t)$). The data record at time $t$ carries information for an (potentially causal) effect of the treatment $A(t)$ on the outcome $Y(t)$, in the context defined by $C_o(t)$. 
The conditional distribution of $O(t)$ is characterized by a conditional distribution of the treatment nodes and the conditional distribution of possibly time-dependent covariates and outcome. An important scenario is that the treatment is sequentially randomized. We are concerned with defining causal effects that can be consistently estimated, with valid inference, for sequentially randomized experiments without further assumptions. More generally, we consider the case when the (possibly causal) effects can be estimated in a double robust manner, analogue to double robust estimation of effects in the i.i.d. causal inference literature. Previous work on the marginal distribution of counterfactual outcomes, such as the marginal distribution of the outcome at a particular time point under a certain intervention on one or more of the treatment nodes, cannot be estimated in a double robust manner \cite{book2018}.  Instead, in this article, we propose a general class of averages of conditional (context-specific) causal parameters that can be estimated in a double robust manner, therefore fully utilizing the sequential randomization. We propose a targeted maximum likelihood estimator (TMLE) of these causal parameters, and present a general theorem establishing the asymptotic consistency and normality of the TMLE. We extend our general framework to a number of typically studied causal target parameters, including a sequentially adaptive design within a single unit that learns the optimal treatment rule for the unit over time. We demonstrate the favorable statistical properties of our estimator through various simulation studies, and provide a software package that implements our methods \cite{malenica2017tstmle}. Our work opens up robust statistical inference for causal questions based on observing a single time-series on a particular unit. 
 \end{abstract}


\section{Introduction}
\subsection{Motivation}

The applications of ``N-of-1'' precision health and medicine are exceedingly important in this era of big data, mobile interventions, and health-monitoring devices. In this manuscript, we are concerned with the development of nonparametric efficient estimators of causal effects of intervention nodes on a subsequent outcome based on observing a single unit over many time points. As such, we address the pressing need for statistical methods that provide actionable inference for a single target unit at any point in time. 

Suppose that one observes a single time series, where at each time $t$ one observes a data record $O(t)$ involving treatment nodes $A(t)$, an outcome node $Y(t)$, and possibly other covariates $L(t)$. In order to talk about causality, we assume that the time-ordering within $O(t)$ with respect to the treatment nodes is known. In most of our examples, $A(t)$ is a single treatment node. but it could also be a vector of time-ordered treatment nodes,  alternated with components of $L(t)$.
We assume that the conditional distribution of $O(t)$ given the observed past is described by a common unknown function $(o(t),C_o(t))\rightarrow\theta(o(t),C_o(t))$ that only depends on the past $O(1),\ldots,O(t-1)$ through a fixed dimensional summary measure $C_o(t)$. For example, one might assume that the conditional density of $O(t)$, given $O(1),\ldots,O(t-1)$, equals a conditional density $\theta(o(t)\mid C_o(t))$ for a common function $\theta$, where this function is otherwise unspecified. More generally, we have that the conditional distribution $P_{O(t)\mid C_o(t),\theta}$ is determined by a common function $\theta$.

The density of $O(t)$ is characterized by the conditional density of treatment nodes and conditional density of the outcome and covariate nodes. One might know, by design, that the conditional density of the treatment node is known (under control of the experimenter) while the other conditional densities are unknown. In that case, one would assume a common conditional density for the outcome and covariate nodes. This setup again describes a model for the distribution of the time-series, indexed by common (in time) conditional densities, analog to the standard conditional stationarity assumptions in time-series literature \cite{timeseries2005}. For certain target parameters  it is also necessary to assume a limited memory in the sense that $C_o(t)$ is only a function of a limited past $O(t-k),\ldots,O(t-1)$ for some fixed dimensional $k$. 

We are interested in models for the probability distribution of the time-series that refrain from making unrealistic parametric assumptions. In particular, we concentrate on models that only make a conditional stationarity assumption. Since the likelihood of the data is parameterized by a function $\theta\in\Theta$,  one can consistently estimate this common function $\theta$ and thereby the probability distribution of the time-series. For example, one might use likelihood based estimation combined with online cross-validation, such as an online super learner \cite{online2014}. We note that standard maximum likelihood estimation would break down for infinite dimensional parameter spaces $\Theta$, due to the curse of dimensionality. 

While possible, our goal is not to estimate the whole mechanism $\theta$ and thereby the whole density of the time-series. We are concerned with statistical inference about causal impacts of the treatment nodes on the outcome nodes, reflecting a certain part of the distribution. For example, one might want to know what the distribution of the outcome at time $\tau$, ($Y(\tau)$), would have been had we intervened on some of the past treatment nodes in the time-series. These type of marginal parameters with the corresponding efficient influence curve and targeted maximum likelihood estimator were developed and proposed in the previous work \cite{book2018}. The asymptotic normality of these estimators relies on consistent estimation (e.g., at an appropriate rate faster than $n^{-1/4}$) of the part of $\theta$ the efficient influence curve depends upon. However, the efficient influence curve of the marginal time-series parameter relies on the whole mechanism $\theta$ in a non-double robust manner  \cite{book2018}. Therefore, even for the situation where the treatment nodes were randomly assigned and known, the inference will still rely on consistent (at rate) estimation of the conditional distributions of the covariate and outcome nodes. This is a stark contrast to the independent and identically distributed case with nonparametric model for the common distribution $\bar{P}$, where the TMLE of such parameters would be completely robust against misspecification of $\bar{P}$ if the treatment mechanism is known. The lack of robustness of the efficient influence function for the marginal time-series parameter is due to its dependence on the density of the marginal distribution of $C_o(t)$ across time $t$, a complex function of the common stationary mechanism $P_{O(t)\mid C_o(t),\theta}$. As such, estimation of the efficient influence curve of the marginal time-series parameter, and thereby the construction of a TMLE, is quite involved and computer intensive \cite{blaauw2017OSL, malenica2017tstmle01}. 

This raises the question if there are causal parameters of the time-series data distribution which are possibly easier to estimate efficiently, and which exhibit robust inference when the treatment mechanism is known. We provide a confirmatory answer in this work. Specifically, we propose a class of statistical target parameters $\bar{\Psi}(\theta)$ defined as the average over time $t$ of $C_o(t)$-specific pathwise differentiable target parameters $\Psi_{C_o(t)}(\theta)$ of the conditional distribution of $P_{O(t)\mid C_o(t),\theta}$. That is, for context $C_o(t)$, one defines a desired target parameter of the distribution of $O(t)$, given $C_o(t)$, as if we were able to observe many observations from this distribution. At that point, we can simply refer to the literature on causal inference, providing statistical estimands $\Psi_{C_o(t)}(\theta)$ of this conditional distribution that can be interpreted as a causal effect of the treatment nodes on the outcome under the randomization assumption (e.g., the $G$-computation formula from the i.i.d. causal inference literature). Interestingly and importantly, one could make the choice $\Psi_{C_o(t)}(\theta)$ of target parameter of the conditional distribution of $O(t)$ given $C_o(t)$ depend on the context $C_o(t)$, allowing one to adapt the choice of target parameter over time in response to $C_o(t)$.

We emphasize that statistical target parameters $\bar{\Psi}(\theta)$ are data-dependent, since they are defined as an average over time of parameters of the conditional distribution of $O(t)$ given the observed realization of $C_o(t)$. As such,  $\bar{\Psi}(\theta)$ depends on the actual realization of the time-series, specifically $(C_o(1),\ldots,C_o(N))$. We also note that since the efficient influence curve $D^*_{C_o(t)}(\theta)$ of each $C_o(t)$-specific target parameter is double robust, it follows that we can estimate the average of $C_o(t)$-specific causal effects in a double robust manner as well. In addition, the linear approximation of the TMLE is a martingale sum $\frac{1}{N}\sum_t D^*_{C_o(t)}(\theta)(O(t))$, allowing for the asymptotic normality of the TMLE to be established based on the martingale central limit theorem and general results for martingale processes \cite{probability2010, handel2009}. 

\subsection{Brief review of relevant literature}
The literature on causal inference in time-series is rapidly growing. The existing statistical methods for performing estimation and inference for causal effect in time-series data are limited, and the literature on this subject has only recently started to develop \cite{Abadie2010, Abadie2015, popescu2013causality, Kleinberg2011ALF, li2017npts, peters2012causalts, poulos2017}. In this review, we reflect on a small subset of this developing literature, focusing on the key aspects and challenges emerging in the statistical estimation of (possibly causal) effects in single time-series data. We emphasize that our review is by no means exhaustive, with many promising methodological approaches not mentioned. 

Granger causality is one of the oldest methods proposed for assessing causal effects in a time-series setting \cite{granger1969}. In short, it quantifies the predictive impact of knowing the past of one time-series for predicting the future of another time-series, given the past of both dependent systems. The idea of Granger causality corresponds to estimating a direct effect with parametric assumptions, where the targets are the entries of the coefficient matrix in a vector autoregressive model. While original Granger causality was restricted to the linear case, more modern formulations include nonlinear Granger causalities \cite{marinazzo2008, faes2011} as well as more general transfer entropy approaches \cite{schreiber2000, darmon2017}. 

The synthetic control method has also become a popular method in social sciences for making causal inferences on observational time-series \cite{Abadie2010,Abadie2015}. Briefly, the method generalizes difference-in-difference estimation in a panel data setting to multivariate time-series with a single treated unit. Similarly to Granger causality, the method relies on parametric assumptions, as well as on availability of pre-period covariates and convex combination of control units. Recent work by Xu proposes a linear fixed effects model that generalizes the synthetic control method to cases of multiple treated units; however, their method still relies on parametric assumptions \cite{xu2017}. On the other hand, Poulos suggests estimating causal effect of a discrete intervention in observational time-series using encoder-decoder neural networks, which does not rely on pretreatment covariates \cite{poulos2017}. Nevertheless, the assumption is that encoder-decoder recurrent neural networks would be able to learn the process for any time-series data, which might not be true in practice. Similarly, Kay et al. advocate for state-space models, and propose inferring causal effects on the basis of a diffusion-regression state-space model that predicts the counterfactual market response in a synthetic control that would have occurred had no intervention had taken place \cite{kay2015}. Other promising approaches include recent work by Li et al. on marginal integration in time-series, which generalizes marginal integration methodology to dependent settings \cite{li2017npts}. Their method provides a nonparametric approach to causality for stationary stochastic processes for a single intervention with no instantaneous effects.  

We also emphasize the significance of our methodology for adaptive randomized trials within a single unit, which are tailored to approximate an optimal treatment rule as sample size grows. Literature on single-unit adaptive sequential trials is almost non-existent to our knowledge, except for the ground-breaking work by Murphy et al \cite{timevarmobile, microrandtrial, liao2015microrand, dempsey2017stratmicrorand, Luers2018, Smith2017}. We aim to build on these ideas in this manuscript, by providing model-free efficient estimators of causal effects based on single-subject interventions on the corresponding unit-level outcome. 

\subsection{Overview of the article}
In Section \ref{sect2} we will formally present the general formulation of the statistical estimation problem, consisting of specifying the statistical model, the target parameter defined as the average of $C_o(t)$-specific target parameters, the corresponding efficient influence curve, and the exact second order expansion of the target parameter around the truth. In addition, in Section \ref{sect2} we formally present the targeted minimum loss estimator (TMLE). Consequently, in Section \ref{sect2b} we present the TMLE analysis resulting in a general theorem establishing asymptotic consistency, asymptotic normality, and asymptotic inference for the time-series setting. In addition, in Section \ref{sect2b} we extend our i.i.d. results for the Highly Adaptive Lasso (HAL) algorithm to our time-series settings. As such, we establish theoretical results that show that the worst-case rate of convergence with time-series HAL is faster than $N^{-1/4}$ under weak conditions. In Section \ref{sect3} we demonstrate estimation of the average of $C_o(t)$-specific causal effects of a single time point intervention $A(t)$ on outcome $Y(t)$ with $O(t)=(A(t),Y(t),W(t))$ being a simple longitudinal data structure with a single treatment node $A(t)$. We generalize methodology described in the previous section in Section \ref{sect4}. In particular, we address the setting of $C_o(t)$-specific causal effect of a multiple time point intervention $A(t,j)$, $j=1,\ldots,K$, on  $Y(t)$, with $O(t)=(A(t:0),L(t:0),\ldots,A(t:K), Y(t)= L(t:K+1))$ being itself a complex longitudinal data structure within a $t$-time-block. In Section \ref{sect5} we return to the simple data structure $O(t)=(A(t),Y(t),W(t))$, but we now aim to learn the optimal individualized treatment rule for $A(t)$ in response to $C_o(t)$ that optimizes the outcome $Y(t)$. Note that the process of generating $A(t)$ is itself controlled by the experimenter, so that one can simultaneously start assigning treatment according to the best current estimate of the optimal treatment rule. Most importantly, we provide model-free methodology for adaptive sequential design learning of the optimal treatment rule within a single unit. In Section \ref{sect6} we show simulation results for the average of $C_o(t)$-specific causal effects of a single time point intervention and adaptive design learning the $C_o(t)$-specific optimal individualized treatment rule parameters. We conclude with a discussion in Section \ref{sect7}.


\section{General approach for robust estimation of averages of context-specific target parameters}\label{sect2}

In this section we present a general formulation of the estimation problem, including the statistical model, target parameter, and second order expansion. With the statistical estimation problem defined, we proceed to the development of the TMLE for the general time-series setup. In the next section we will analyze the TMLE to prove a general theorem establishing asymptotic consistency and normality. 

\subsection{Statistical estimation problem}
{\bf Data and likelihood:}
Let the observed data be $O(t)$, $t=1,\ldots,N$.  We note that $O(t)$ is of a fixed dimension in time $t$, and is an element of a Euclidean set ${\cal O}$. An important case is that $O(t)=(A(t),Y(t),W(t))$, where $A(t)$ occurs before $Y(t)$ and $Y(t)$ occurs before $W(t)$. Let $A(t)$ denote the exposure or treatment, while $Y(t)$ denotes an outcome of interest at time $t$, with $W(t)$ being all the other post-treatment measurements beyond the outcome of interest. We define $O^N=(O(t):t=1,\ldots,N)$, and let $P^N$ denote its probability measure. In this example we can factorize the probability density of the data according to the time ordering as follows:
\begin{align*}
p^N(o) &= \prod_{t=1}^{N}p_{a(t)}(a(t) | \bar{o}(t-1))  \prod_{t=1}^{N}p_{y(t)}(y(t) | \bar{o}(t-1),a(t)) \\ 
&\phantom{{}=0} \prod_{t=1}^{N}p_{w(t)}(w(t) | \bar{o}(t-1),y(t),a(t)) 
\end{align*}
Here, $p_{a(t)}$, $p_{y(t)}$ and $p_{w(t)}$ denote the conditional probability densities of $A(t)$, $Y(t)$ and $W(t)$ given the relevant past. We define $\mu_a$, $\mu_y$ and $\mu_w$ as the corresponding dominating measures. Finally, let $P_{O(t)\mid \bar{O}(t-1)}$ be the conditional probability distribution of $O(t)$, given $\bar{O}(t-1)$, defined on a sigma-algebra of ${\cal O}$.

\vspace{2mm}
\noindent
{\bf Statistical model for time-series:}
Since $O^N$ represents a single time-series, a dependent process, we observe only a single draw from $P^N$. As a result, we are unable to estimate $P^N$ from this single observation without additional assumptions. In particular, we assume that $P_{O(t)\mid \bar{O}(t-1)}$ depends on $\bar{O}(t-1)$ through a summary measure $C_o(t)=C_o(\bar{O}(t-1))\in {\cal C}$ of fixed dimension. For later notational convenience, we denote this conditional distribution with $P_{C_o(t)}$. Then, the density $p_{C_o(t)}$ of $P_{C_o(t)}$ with respect to a dominating measure $\mu_{C_o(t)}$ is a conditional density $(o,C_o)\rightarrow p_{C_o(t)}(o\mid C_o)$ so that for each value of $C_o(t)$, $\int p_{C_o(t)}(o\mid C_o(t))d\mu_{C_o(t)}(o)=1$. Additionally, we assume that $p_{C_o(t)}$ is parameterized by a common (in time $t$) function $\theta\in \Theta$, where $\theta:{\cal C}\times {\cal O}\rightarrow\openr$ is a function 
$(c,o)\rightarrow \theta(c,o)$. We note that $p_{C_o(t)}$ only depends on $\theta$ through $\theta_{C_o(t)}\equiv \theta(C_o(t),\cdot)$. In some of our examples, we simply assume that
$(c,o)\rightarrow p_{C_o(t)}(o\mid c)=\bar{p}(o\mid c)$ is constant in $t$, so that $\theta=\bar{p}$ is a common conditional density. The simplified version of the assumption is generally known as the conditional (strong) stationarity assumption. On the other hand, since $p_{C_o(t)}$ factors into multiple conditional densities, we note that there are many examples for which some parts of $p_{C_o(t)}$ might be assumed to be known and change over time $t$, while others play no role in the estimation and can therefore be unrestricted. As such, we emphasize that the key factors of $p_{C_o(t)}$ necessary for the estimation procedure must satisfy the stationarity assumption. For example, the conditional density of treatment node $A(t)$ might be known and could change over time $t$, whereas other parts of the likelihood are assumed to be unknown but constant in time $t$. 

\vspace{2mm}
\noindent
This defines a statistical model ${\cal M}^N=\{P^N_{\theta}:\theta\}$ where $P^N_{\theta}$ is the probability measure for the time-series implied by 
$p_{C_o(t)}=p_{\theta,C_o(t)}$. Additionally, we define a statistical model of distributions of $O(t)$ at time $t$, conditional on realized summary $C_o(t)$. 
In particular, let ${\cal M}(C_o(t))=\{P_{\theta,C_o(t)}:\theta\}$ be the model for $P_{C_o(t)}$ for a given $C_o(t)$ implied by ${\cal M}^N$. 

\vspace{2mm}
\noindent
{\bf Target parameter:}
First, we define a target parameter conditional on realized summary $C_o(t)$.  For a given $C_o(t)$, we define a target parameter $\Psi_{C_o(t)}:{\cal M}(C_o(t))\rightarrow\openr$ that is pathwise differentiable with canonical gradient $D^*_{C_o(t)}(\theta)$ at $P_{\theta,C_o(t)}$  in ${\cal M}(C_o(t))$. We remind that the variance of the efficient influence curve (canonical gradient) gives the generalized Cramer-Rao lower bound for the variance of any regular asymptotically linear estimator  based on observing $n$ i.i.d. observations from $P_{\theta,C_o(t)}$ \cite{vaart2000}. Let $\Psi_{C_o(t)}(\theta)=\Psi_{C_o(t)}(P_{\theta,C_o(t)})$ so that we can also use the notation $\Psi_{C_o(t)}(\theta)$ to denote the target parameter. By assumption, $\Psi_{C_o(t)}(\theta)$ only depends on $\theta$ through its section  $o\rightarrow \theta_{C_o(t)}(o)=\theta(o,C_o(t))$. We also denote this collection of $C_o(t)$-specific canonical gradients with a single function $(c,o)\rightarrow D^*(\theta)(c,o)$ so that $D^*_{C_o(t)}(\theta)(o)=D^*(\theta)(C_o(t),o)$, viewing it as a function of $(C_o(t),O(t))$.  Note that, for a given $C_o(t)$, this canonical gradient $D^*_{C_o(t)}(\theta)$ is a function of $O\in {\cal O}$ which has conditional mean zero w.r.t. $P_{\theta,C_o(t)}$.

\vspace{2mm}
\noindent
Additionally, we define an  average over time of these $C_o(t)$-specific target parameters. In particular, we define the following target parameter  $\Psi^N:{\cal M}^N\rightarrow\openr$ of the data distribution $P^N\in {\cal M}^N$, which is a function of $\theta$:
\begin{eqnarray*}
\Psi^N(P^N)&=&\bar{\Psi}(\theta)\equiv \frac{1}{N}\sum_{t=1}^N \Psi_{C_o(t)}(P_{\theta,C_o(t)})\\
&=&\frac{1}{N}\sum_{t=1}^N \Psi_{C_o(t)}(\theta).
\end{eqnarray*}
We note that $\bar{\Psi}(\theta)$ is a data dependent target parameter since its value depends on the realized $C_o(t)$, $t=1,\ldots,N$.

\paragraph{Remark:}
An important variation of the above formulation is to select $\Psi_{C_o(t)}:{\cal M}(C_o(t))^{np}\rightarrow\openr$ on a  more nonparametric model ${\cal M}(C_o(t))^{np}$ instead of ${\cal M}(C_o(t))$, while still utilizing the actual model ${\cal M}^N$ when we estimate the unknown $\theta$. That is, even though we assumed $\theta\in \Theta$, we might define
${\cal M}(C_o(t))^{np}=\{p_{\theta,C_o(t)}:\theta\in\Theta^{np}\}$ for some larger set $\Theta^{np}$. Although this will affect the efficiency of the estimator of our desired $\Psi^N(P^N)$, it might result in more robust estimators (e.g., ones that remain consistent as long as the conditional probability of treatment nodes are consistently estimated), due to its canonical gradient having a double robust structure. 
For example, $O(t)=(L(t,j),A(t,j):j=1,\ldots,K)$ might have $K$ treatment and $K$ time-dependent covariate nodes $A(t,j)$ and $L(t,j)$, respectively, ordered in time $j$ within the time block indicated by $t$. The model ${\cal M}^N$ might assume stationarity and/or Markov assumption in $j$ within the time-points of time-block $t$,  but we might nonetheless define the parameter $\Psi$ on the model ${\cal M}(C_o)^{np}$ that ignores these stationarity assumptions within the time-block $t$, while preserving conditional stationarity in $t$ of the density of $O(t)$, given $C_o(t)$.
\vspace{0.1in}


\subsection{Defining the TMLE}
We refer the reader to our Targeted Learning books and articles for a detailed description of the theory of TMLE \cite{book2018, tl2011, tmle2006}. 
We note that TMLE is a two-step procedure where one first obtains an estimate of the data-generating distribution, or the relevant parts of the data-generating distribution. The second stage updates the initial fit in a step targeted towards making an optimal bias-variance tradeoff for the parameter of interest, instead of the whole density. 
Following the empirical process literature, we define $P_N f$ to be the empirical average of function $f$, and $Pf = \mathbb{E}_Pf(O)$.

\vspace{2mm}
\noindent
{\bf Loss function:}
For a given $C_o(t)$, let $L_{C_o(t)}(\theta)(O(t))$ be a loss function for $\theta_{C_o(t)}$, defined such that we have the following:
\[
P_{\theta_0,C_o(t)}L_{C_o(t)}(\theta_{0,C_o(t)})=\min_{\theta_{C_o(t)}}P_{\theta_0,C_o(t)}L_{C_o(t)}(\theta_{C_o(t)})
\]
Therefore, given $C_o(t)$,  the true $\theta_{0,C_o(t)}$ minimizes the risk $P_{\theta_0,C_o(t)}L_{C_o(t)}(\theta_{C_o(t)})$ under $P_{\theta_0,C_o(t)}$. For notational simplicity, we will also denote this loss function with $L(\theta)$, but then it is viewed as a function $L(\theta):{\cal O}\times {\cal C}\rightarrow\openr$ of $(C_o(t),O(t))$, so that
$L_{C_o(t)}(\theta_{C_o(t)})(o)=L(\theta)(o,C_o(t))$.
For example, one might consider the log-likelihood loss function, 
 \[
 L_{C_o(t)}(\theta)(O(t))= L(\theta)(C_o(t),O(t)) = -\log p_{\theta,C_o(t)}(O(t))
 \]

\vspace{2mm}
\noindent
{\bf Least favorable submodel through initial estimator:}
Let $\theta_N$ be an initial estimator of $\theta_0$.
For a $\theta_N$ in our statistical model, we define a parametric working model $\{\theta_{N,\epsilon}:\epsilon\}$ through $\theta_N$ with finite-dimensional parameter $\epsilon$ so that $\epsilon = 0$ denotes $\theta_N$.  In particular, we define a parametric family of fluctuations of the initial estimator $\theta_N$ of $\theta_0$ with fluctuation parameter $\epsilon$, along with an appropriate loss function, so that the linear combination of the components of the derivate of the loss evaluated at $\epsilon = 0$ spans the efficient influence curve at the initial estimator. That is,  for a given $\theta_N$,  $\{\theta_{N,\epsilon}:\epsilon\}\subset \Theta$ is  a submodel of $\Theta$ through $\theta_N$ such that:
\begin{equation}\label{lfm}
\left \langle \left  .\frac{d}{d\epsilon}L(\theta_{N,\epsilon})\right |_{\epsilon =0} \right \rangle \supset D^*(\theta_N)
\end{equation}
where we used the notation $\langle S\rangle $ for the linear span of the components of the function $S$.
This is equivalent with stating that for each $C_o(t)$,  $\{\theta_{N,C_o(t),\epsilon}:\epsilon\}$ is a submodel through $\theta_{N,C_o(t)}=\theta_N(C_o(t),\cdot)$ at $\epsilon =0$, such that we have the following:
\[
\left \langle \left .  \frac{d}{d\epsilon}L_{C_o(t)}(\theta_{N,C_o(t),\epsilon}) \right |_{\epsilon =0}\right \rangle \supset D^*_{C_o(t)}(\theta_N)
\]
Therefore, for each $c\in {\cal C}$ and target parameter $\Psi_c:{\cal M}(c)\rightarrow\openr$, we have that $\{P_{\theta_{N,c,\epsilon},c}:\epsilon\}$ is a local least favorable submodel. Applied to $c=C_o(t)$ at a particular time $t$, this states that $\{P_{\theta_{N,C_o(t),\epsilon},C_o(t)}:\epsilon\}$ is a local least favorable submodel through $P_{\theta_{N,C_o(t)},C_o(t)}\in {\cal M}(C_o(t))$. 

\vspace{2mm}
\noindent
Alternatively, we could also define a universal least favorable submodel so that for each $C_o(t)$ and for all appropriate $\epsilon$,
\[
\frac{d}{d \epsilon}L_{C_o(t)}(\theta_{N,C_o(t), \epsilon}) =D^*_{C_o(t)}(\theta_{N,C_o(t), \epsilon})
\]
and equivalently, for all $\epsilon$ we have that:
\begin{equation}\label{ulfm}
\frac{d}{d\epsilon}L(\theta_{N,\epsilon}) =D^*(\theta_{N,\epsilon}) 
\end{equation}

\vspace{2mm}
\noindent
{\bf TMLE-update step:}
We define $\theta_N$ as an initial estimate of $\theta_0$, achieved using the super learner methodology based on the previously defined loss $\sum_t L(\theta)(C_o(t),O(t))$ and one of the appropriate cross-validation schemes for dependent settings \cite{sl2007, onlinecv2018}. 
 Given the initial estimator $\theta_N$ of $\theta_0$, we compute the minimum loss estimator (MLE) of $\epsilon$, given by:
\[
\epsilon_N=\arg\min_{\epsilon}\sum_t L(\theta_{N,\epsilon})(C_o(t),O(t))
\] 
If one uses a universal least favorable submodel, then, by (\ref{ulfm}), the score equation of this MLE yields:
\[
\sum_t D^*(\theta_{N,\epsilon_N})(C_o(t),O(t)) = 0
\]

\noindent
On the other hand, if one uses a local least favorable submodel, the updating process will need to be iterated. In particular, for $k=0$, let $\theta_N^0=\theta_N$ be the initial estimate of $\theta_0$. Similarly as before, we compute the MLE of $\epsilon$ at $\theta_N^k$ for as many $k$ as necessary:
\[
\epsilon_N^k=\arg\min_{\epsilon}\sum_t L(\theta_{N,\epsilon}^k)(C_o(t),O(t))
\]
we note that for $k=1$, this yields the first step TMLE, $\theta_N^1 =  \theta_{N,\epsilon_N^0}^0$. We iterate by $k\mapsto k+1$ and repeat the updating step until $\epsilon_N^k\approx 0$. The final update, denoted as $\theta_N^*$, is the TMLE of $\theta_0$. By (\ref{lfm}), the iterative TMLE also solves the efficient influence function estimating equation as follows:
\[
\sum_t D^*(\theta_N^*)(C_o(t),O(t))\approx 0
\]
\noindent
Below we define the second order remainder of a Taylor expansion of the target parameter $\Psi_{C_o(t)}$ at $P_{\theta_0,C_o(t)}$, and denote it with $R_{2,C_o(t)}(\theta_N,\theta_0)$. We conjecture that, under regularity conditions, if the initial estimator $\theta_N$ is consistent at a good rate  so that $\frac{1}{N}\sum_t R_{2,C_o(t)}(\theta_N,\theta_0)=o_P(N^{-1/2})$, then the efficient score equation will be solved in one step using the local least favorable submodel up until an $o_P(N^{-1/2})$ term. Such a result has been proved for the i.i.d. case in \citep{vanderLaan2018haltmle, vanderlaan2017efficTMLEwhal}. Therefore, given a good initial estimator, few iterations will be needed to approximately solve the efficient score equation. We define $\theta_N^*$ to be the one-step TMLE or the final update for the iterative TMLE. Then the efficient influence curve evaluated at $\theta_N^*$ has the following property:
\begin{equation}\label{g2}
\sum_t D^*(\theta_N^*)(C_o(t),O(t))=o_P(N^{-1/2})
\end{equation}

\noindent
Finally, the TMLE of $\bar{\Psi}(\theta_0)$ is given by the plug-in estimator $\bar{\Psi}(\theta_N^*)$. The efficient score equation (\ref{g2})  provides the basis for establishing asymptotic linearity and efficiency of $\bar{\Psi}(\theta_N^*)$, as carried out in the next section.

\section{Analysis of the TMLE}\label{sect2b}


\noindent
We define the following exact second order expansion for the $C_o(t)$-specific target parameter:
\begin{equation}\label{g1}
\Psi_{C_o(t)}(P_{\theta,C_o(t)}) - \Psi_{C_o(t)}(P_{\theta_0,C_o(t)}) =  (P_{\theta,C_o(t)} -P_{\theta_0,C_o(t)}) D^*_{C_o(t)}(\theta) + R_{2,C_o(t)}(\theta,\theta_0),
\end{equation}
where the remainder is defined as:
\begin{align*}
R_{2,C_o(t)}(\theta,\theta_0) &\equiv \Psi_{C_o(t)}(P_{\theta,C_o(t)}) - \Psi_{C_o(t)}(P_{\theta_0,C_o(t)}) + P_{\theta_0,C_o(t)} D^*_{C_o(t)}(\theta)
\end{align*}
For any setting one can verify that indeed $R_{2,C_o(t)}$ represents a second order difference between $P_{\theta_N,C_o(t)} $ and $P_{\theta_0,C_o(t)}$ for the given $C_o(t)$, which is a natural consequence of the pathwise differentiability of $\Psi_{C_o(t)}$. We emphasize that $P_{\theta,C_o(t)}D^*_{C_o(t)}(\theta)=0$. Combining the efficient score equation (\ref{g2}) with the second order expansion (\ref{g1}) of $\Psi_{C_o(t)}$, with $\theta$ being the TMLE $\theta_N^*$ we have the following exact second order expansion for our TMLE:
\begin{equation}\label{g1a}
\begin{aligned}
\frac{1}{N} \sum_t (\Psi_{C_o(t)}(\theta_N^*) - \Psi_{C_o(t)}(\theta_0)) &= \frac{1}{N} \sum_t (D^*(\theta_N^*)(C_o(t),O(t)) -P_{\theta_0,C_o(t)}D^*(\theta_N^*)) \\
&\phantom{{}=0} + \frac{1}{N}\sum_t R_{2,C_o(t)}(\theta_N^*,\theta_0)+o_P(N^{-1/2})
\end{aligned}
\end{equation}
We note that the left-hand side is also denoted with $\bar{\Psi}(\theta_N^*)-\bar{\Psi}(\theta_0)$, as defined in the previous section. 

\vspace{2mm}
\noindent
The leading term in the above expansion can be denoted as $M_N(\theta_N^*)$, for a Martingale process $(M_N(\theta):\theta)$ evaluated at $\theta_N^*$. In general, weak convergence of a process $(N^{1/2}M_N(\theta):\theta)$  to a Gaussian process  is equivalent with convergence of all finite dimensional distributions $N^{1/2}(M_N(\theta_1),\ldots,M_N(\theta_k))$ for a vector $(\theta_1,\ldots,\theta_k)$ and an asymptotic equicontinuity/tightness condition. The convergence of the finite dimensional distributions is immediately implied by the multivariate martingale central limit theorem. 
Asymptotic equicontinuity is typically defined as a statement that $M_N(\theta_N)-M_N(\theta^*)=o_P(N^{-1/2})$ if $d_N(\theta_N,\theta^*)\rightarrow_p 0$ for a specified dissimilarity $d_N()$ and limit $\theta^*$. 
This type of asymptotic equicontinuity allows us then to approximate $M_N(\theta_N^*)=M_N(\theta)+(M_N(\theta_N^*)-M_N(\theta^*))$ with
$M_N(\theta^*)+o_P(N^{-1/2})$. 

\vspace{2mm}
\noindent
Let ${\cal F}$ be a class of multivariate real valued functions of $(O,C)\in {\cal C}\times  {\cal O}$. Suppose that $ D^*(\theta_N^*)\in {\cal F}$ with probability tending to 1. We consider a martingale process $(M_N(f):f\in {\cal F})$ indexed by this class of functions ${\cal F}$  defined by: 
\[
M_N(f)= \frac{1}{N} \sum_{t=1}^N \{f(C_o(t),O(t))-P_{\theta_0,C_o(t)}f\}
\]
We note that for all $f \in \cal F$,  $N M_N(f)$ is a discrete martingale in $N$. In our application, we have  ${\cal F}=\{D^*(\theta):\theta\in \Theta\}$. To establish the asymptotic equicontinuity, we could rely on a maximal inequality for martingales, Proposition A.2 in Handel, used in  \cite{handel2009, chambaz2010} to prove that $M_N(f_N)-M_N(f^*)=o_P(N^{-1/2})$ if $1/N\sum_t P_{\theta_0,C_o(t)}\{f_N-f^*\}^2(C_o(t),O(t))\rightarrow_p 0$. However, this maximal inequality would rely on the class of functions ${\cal F}$ to have a finite entropy integral, as defined below, with respect to the bracketing entropy. Since we want to allow that ${\cal F}$ contains all cadlag functions with a sectional variation norm bounded by a universal constant, which is a convex hull of indicator functions, we only have that the entropy integral w.r.t. covering number is bounded. Fortunately, \citep{chambaz2017tsoit} establish this desired asymptotic equicontinuity for classes of functions ${\cal F}$ for which the regular entropy integral is bounded. 

\vspace{2mm}
\noindent
Let $N(\epsilon, {\cal F}, L^2(P))$ denote the covering number, defined as the number of balls of size $\epsilon$ needed to cover $\cal F$ embedded in Hilbert space $L^2(P)$. We impose the entropy conditions on $\cal F$ such that:
\begin{equation}\label{entropyintegral}
\int_0^1 \sup_P \sqrt{\log N(\epsilon, \mathcal{F}, L^2(P) ) d \epsilon} < \infty
\end{equation}



\vspace{3mm}
\noindent
\begin{lemma}\label{lemmaasympequi}
\textbf{Asymptotic Equicontinuity of a Martingale Process} \\
\textit{Let $\frac{1}{N} \sum_t P_{\theta_0,C_o(t)} \{f_N(C_o(t),O(t)) - f^*(C_o(t),O(t)) \}^2 \xrightarrow{p} 0$. Under the above entropy condition (\ref{entropyintegral}) on ${\cal F}$, $M_N()$ is asymptotically equicontinuous w.r.t. a dissimilarity measure so that: 
\[M_N(f_N)-M_N(f^*)=o_P(N^{-1/2})
\] where $f^*(C_o(t),O(t))$ is the limit for $f_N(C_o(t),O(t))$.}
 \end{lemma}
 
\noindent
The proof of Lemma \ref{lemmaasympequi} is given in \citep{chambaz2017tsoit}. An application of Lemma \ref{lemmaasympequi} to our martingale process $(M_N(\theta):\theta)$ and corresponding class ${\cal F}=\{D^*(\theta):\theta\in \Theta\}$ provides us with consistency condition (C2 below) on $\theta_N^*$ so that $M_N(\theta_N^*)=M_N(\theta^*)+o_P(N^{-1/2})$.

\vspace{3mm}
\noindent
\textbf{Asymptotic linearity or negligibility of the remainder term:}  We assume that the limit  $\theta^*$ satisfies $R_{2,C_o(t)}(\theta^*,\theta_0)=0$ for all $C_o(t)$. Then, the remainder $\frac{1}{N}\sum_t R_{2,C_o(t)}(\theta_N^*,\theta_0)$
represents a term that converges to zero. Additionally, we assume that this remainder term can be represented by a martingale sum $M_N(f_1)$ for some $f_1=f_1(\theta^*,\theta_0)$ plus a second order term that is $o_P(N^{-1/2})$, where $$M_N(f_1)=\frac{1}{N} \sum_t \{f_1(O(t),C_o(t))-P_{0,C_o(t)}f_1\}.$$

\noindent
We note that a special case for Condition 3 corresponds to $\theta^*=\theta_0$, in which case this assumption is equivalent to assuming $\frac{1}{N}\sum_t R_{2,C_o(t)}(\theta_N^*,\theta_0)=o_P(N^{-1/2})$. Under the entropy, consistency, and the above condition on the remainder term it follows that:
\begin{eqnarray*}
\bar{\Psi}(\theta_N^*)-\bar{\Psi}(\theta_0) &=&
\frac{1}{N}\sum_t\left\{ D^*(\theta^*)(C_o(t),O(t))-P_{\theta_0,C_o(t)}D^*(\theta^*) \right\}\\
&&+\frac{1}{N}\sum_t\left\{f_1(C_o(t),O(t))-P_{\theta_0,C_o(t)}f_1\right\}+o_P(N^{-1/2})
\end{eqnarray*}
The right-hand side is a discrete martingale up until an $o_P(N^{-1/2})$, whose standardized version converges to a normal limit distribution. 

\vspace{2mm}
\noindent
To summarize the proof above, we formally state Theorem \ref{thgen}, establishing asymptotic normality for the averages of context-specific target parameters under conditions (C1-C4):

\vspace{2mm}
\noindent
\begin{enumerate}
\item Define $M_N(\theta)$ as $M_N(\theta)=\frac{1}{N}\sum_t D^*(\theta)(C_o(t),O(t))-P_{\theta_0,C_o(t)}D^*(\theta)$. 
Let ${\cal F}=\{D^*(\theta):\theta\in \Theta\}$ and assume the entropy condition $$\int_0^1 \sup_P \sqrt{\log N(\epsilon, \mathcal{F}, L^2(P) ) d \epsilon} < \infty .$$ 
Then the martingale process $M_N(\theta)$ indexed by $\theta\in \Theta$ is asymptotic equicontinuous in the following sense: $M_N(\theta_N^*)-M_N(\theta^*)=o_P(N^{-1/2})$ if $\frac{1}{N} \sum_t P_{\theta_0,C_o(t)}\{D^*(\theta_N^*)-D^*(\theta^*)\}^2\rightarrow_p 0$.
\item For a limit $\theta^*$ satisfying $R_{2,C_o(t)}(\theta^*,\theta_0)=0$ with probability 1, $\frac{1}{N} \sum_t P_{\theta_0,C_o(t)}\{D^*(\theta_N^*)-D^*(\theta^*)\}^2\rightarrow_p 0$. 
\item Assume $\frac{1}{N}\sum_t R_{2,C_o(t)}(\theta_N^*,\theta_0)=M_n(f_1)+o_P(N^{-1/2})$ for some $f_1=f_1(\theta^*,\theta_0)$ and martingale process $M_N(f)=\frac{1}{N} \sum_t \{f(O(t),C_o(t))-P_{\theta_0,C_o(t)}f\}$.
\item Let $\bar{f}=D^*(\theta^*)+f_1(\theta^*,\theta_0)$. We assume $\frac{1}{N}\sum_{t=1}^N P_{0,C_o(t)}\bar{f}^2\rightarrow \sigma^2_0$ 
as $N\rightarrow\infty$ a.s.
\end{enumerate}

\noindent
\begin{theorem}[Asymptotic normality of TMLE]\label{thgen}\ 
Let $\theta_N^*$ be the one-step TMLE or the iterative TMLE, so that $\sum_t D^*(\theta_N^*)(C_o(t),O(t))=o_P(N^{-1/2})$. Further, assume that $D^*(\theta^*_N) \in \cal F$ with probability tending to 1, for a class of functions $\cal F$ satisfying the entropy integral condition (\ref{entropyintegral}) of C1 above. If also C2, C3 and C4 hold, then:
\vspace{3mm}
\[
\sqrt{N}(\bar{\Psi}(\theta_N^*)-\bar{\Psi}(\theta_0))\Rightarrow N(0,\sigma^2_0)
\] 

\vspace{3mm}
\noindent
where  $\sigma^2_0$ is the limit of $\frac{1}{N} \sum_t \bar{f}^2(C_o(t),O(t))$ and $\bar{f}=D^*(\theta^*)+f_1(\theta^*,\theta_0)$.
\end{theorem}

\vspace{3mm}
\begin{proof} For completeness, we present here the formal proof, following the steps above.
The definition of $\theta_N^*$ combined with C1 and C2 yield that:
\begin{eqnarray*}
\bar{\Psi}(\theta_N^*)-\bar{\Psi}(\theta_0)&=&
\frac{1}{N}\sum_t\left\{ D^*(\theta^*)(C_o(t),O(t))-P_{\theta_0,C_o(t)}D^*(\theta^*)\right\}\\
&&\hspace*{+0.5cm}
+ \  o_P(N^{-1/2}) + \frac{1}{N} \sum_t R_{2,C_o(t)}(\theta^*_N, \theta_0)
\end{eqnarray*}

\noindent
Further, by C3 we can see that:
\begin{eqnarray*}
\bar{\Psi}(\theta_N^*)-\bar{\Psi}(\theta_0)&=&
\frac{1}{N}\sum_t\left\{ D^*(\theta^*)(C_o(t),O(t))-P_{\theta_0,C_o(t)}D^*(\theta^*)\right\}\\
&&\hspace*{-3cm}
+\frac{1}{N}\sum_t\left\{f_1(\theta^*,\theta_0)(C_o(t),O(t))-P_{\theta_0,C_o(t)}f_1(\theta^*,\theta_0)\right\}+o_P(N^{-1/2})
\end{eqnarray*}

\noindent
The sum of the two terms on the right hand side is a  discrete martingale $M_N(\bar{f})$ with $\bar{f}$ defined as $\bar{f}=D^*(\theta^*)+f_1(\theta^*,\theta_0)$. By the martingale central limit theorem, $\bar{f}$ converges to a centered Gaussian distribution  with covariance $\lim_{t \rightarrow \infty}\frac{1}{N} \sum_t \bar{f}^2(C_o(t),O(t))$.
\end{proof}

\subsection{Class of cadlag functions with uniformly bounded sectional variation norm}
An important class for which the entropy integral (\ref{entropyintegral}) of condition C1 is finite is the class of multivariate real valued cadlag functions on $[0,\tau]\subset\openr^k$ with a uniform bound on its finite sectional variation norm defined by: 
\[
\pl f\pl_v^* \ \equiv \ \mid f(0)\mid+\sum_{s\subset\{1,\ldots,k\}}\int_{(0_s,\tau_s]} \mid df_s(u_s)\mid \]
Here, $f_s$ is the $s$-specific section $f_s(x)=f(x_s,0_{-s})$ that sets the coordinates in the complement of $s$ equal to $0$, and $df_s(u_s)$ denotes integration w.r.t. measure generated by this $s$-specific section on $(0_s,\tau_s]$. 
The sum is over all subsets $s$ of $\{1,\ldots,k\}$.
Any such function for which $\pl f\pl_v^* < \infty$ can be represented as $f(x)=f(0)+\sum_{s\subset\{1,\ldots,k\}}\int_{(0_s,x_s]} df_s(u_s)$ \cite{gill1995}. The latter representation theorem shows that this class of functions is a convex hull of the indicator class $\{X_s\rightarrow I(X_s>u_s): u_s,s\}$, and a fundamental theorem in \citep{vanderVaart&Wellner96} shows that a convex hull of a Donsker class is a Donsker class itself, where Donsker class can be defined as a class of functions for which the entropy integral (\ref{entropyintegral}) is finite. 

\subsection{The HAL-MLE as initial estimator}
Let $L(\theta)(O(t),C_o(t))$ be a loss function for $\theta$ so that $\theta_0=\arg\min_{\theta\in \Theta} \frac{1}{N} \sum_{t=1}^N P_{\theta_0,C_o(t)}L(\theta)$.
Suppose that $\Theta$ is contained in a class of multivariate real valued cadlag functions on a cube $[0,\tau]$ with sectional variation norm bounded by a universal constant $C^u<\infty$.
We assume that for each $\theta$, $L(\theta)$ is a multivariate real valued cadlag function on a cube $[0,\tau_o]$ whose sectional variation norm can be bounded by the sectional variation norm of $\theta$ in the sense that $\sup_{\theta\in \Theta}\pl L(\theta)\pl_v^*/\pl \theta\pl_v^*<\infty$.
Then, $\{L(\theta):\theta\in\Theta\}$ is contained in a class of multivariate real valued cadlag functions on cube $[0,\tau_o]$ bounded by a universal constant $C^u_o<\infty$.
Let $\Theta(C)=\{\theta\in \Theta:\pl \theta\pl_v^*\leq C\}$ be a constrained subset of $\Theta$ by constraining the sectional variation norm to a number $C$ smaller or equal than the known upper-bound $C^u$.
Consider the $C$-specific MLE defined as:
\[
\theta_{C,n}=\arg\min_{\theta\in \Theta(C)}\frac{1}{N}\sum_{t=1}^N L(\theta)(O(t),C_o(t))\]
Further, let $\theta_{C,0}=\arg\min_{\theta\in \Theta(C)}\frac{1}{N}\sum_{t=1}^N P_{\theta_0,C_o(t)}L(\theta)$. We define loss-based dissimilarity implied by loss-function  $L(\theta)$ as:
\[
d_{0,N}(\theta,\theta_0)=\frac{1}{N}\sum_{t=1}^N P_{\theta_0,C_o(t)}\{L(\theta)-L(\theta_0)\}\]
For further notational convenience, we will also use the notation $L(\theta,\theta_0)=L(\theta)-L(\theta_0)$.


\begin{theorem}[Minimum loss-based estimator in class of cadlag functions with finite variation norm]\label{thhal}
Let $L(\theta)(O_t,C_o(t))$ be a loss function for $\theta$, and $\Theta(C)=\{\theta\in \Theta:\pl \theta\pl_v^*\leq C\}$ be the set of cadlag functions with variation norm smaller than $C$. We define $d_{0,N}(\theta,\theta_0)$ as the loss-based dissimilarity for $L(\theta)(O_t,C_o(t))$. If:
\begin{enumerate}
\item $\sup_{\theta\in \Theta}\pl L(\theta)\pl_v^*/\pl \theta\pl_v^*<\infty$;
\item $d_{0,N}(\theta_{C,N},\theta_{C,0})\rightarrow_p 0$ implies $\frac{1}{N}\sum_{t=1}^N P_{\theta_0,C_o(t)}L(\theta_{C,N},\theta_{C,0})^2\rightarrow_p 0$,
\end{enumerate}
then $d_{0,N}(\theta_{C,N},\theta_{C,0})=o_P(N^{-1/2})$.
\end{theorem}

\vspace{2mm}
\noindent

\begin{proof}
Let $M_N(\theta)=\frac{1}{N}\sum_{t=1}^N \{ L(\theta)-P_{\theta_0,C_o(t)}L(\theta)\}$ define the martingale process $M_N(\theta)$. Note that $M_N(\theta)$ is a martingale process indexed by a class of multivariate real valued cadlag functions with a uniform bound on the sectional variation norm. Its asymptotic equicontinuity is presented in Lemma \ref{lemmaasympequi} stated previously. It follows that:
\begin{eqnarray*}
0&\leq& d_{0,N}(\theta_{C,N},\theta_{C,0})\\
&=& \frac{1}{N}\sum_{t=1}^N P_{\theta_0,C_o(t)}L(\theta_{C,N},\theta_{C,0})\\
&=&-\frac{1}{N}\sum_{t=1}^N \left\{ L(\theta_{C,N},\theta_{C,0})-P_{\theta_0,C_o(t)}L(\theta_{C,N},\theta_{C,0})\right\}\\
&&+\frac{1}{N}\sum_{t=1}^N L(\theta_{C,N},\theta_{C,0})\\
&\leq&-\frac{1}{N}\sum_{t=1}^N \left\{ L(\theta_{C,N},\theta_{C,0})-P_{\theta_0,C_o(t)}L(\theta_{C,N},\theta_{C,0})\right\}\\
&\equiv&-\left\{ M_N(\theta_{C,N})-M_N(\theta_{C,0})\right\},
\end{eqnarray*}
where the first and second inequality both follow from the definition of $\theta_{C,N}$ as the minimizer of $\frac{1}{N}\sum_{t=1}^N L(\theta)(O(t),C_o(t))$ over all $\theta\in \Theta(C)$. We know that $\sup_{\theta\in \Theta(C)}\mid M_N(\theta)\mid =O_P(N^{-1/2})$, and if $\frac{1}{N}\sum_{t=1}^N P_{\theta_0,C_o(t)} \{L(\theta_N,\theta_0)\}^2 \rightarrow_p 0$, then we have that $M_N(\theta_N)-M_N(\theta_0)=o_P(N^{-1/2})$ by Lemma \ref{lemmaasympequi}.
The first statement proves that $d_{0,N}(\theta_{C,N},\theta_{C,0})=O_P(n^{-1/2})$. By assumption, we also have that  $d_{0,N}(\theta_{C,N},\theta_{C,0})=o_P(1)$ implies $\frac{1}{N}\sum_{t=1}^N P_{\theta_0,C_o(t)} \{L(\theta_{C,N},\theta_{C,0})\}^2 \rightarrow_p 0$. Therefore by asymptotic equicontinuity of $M_N(\theta)$ we have that $M_N(\theta_{C,N})-M_N(\theta_{C,0})=o_P(N^{-1/2})$. This proves that $d_{0,N}(\theta_{C,N},\theta_{C,0})=o_P(N^{-1/2})$. 
\end{proof}

\noindent
We note that using the exact entropy bound for the class of multivariate cadlag functions with uniformly bounded sectional variation norm we can also derive a more precise rate of $o_P(N^{-1/2-\alpha(d_o)})$, for an $\alpha(d)$ that behaves as $1/d$ and $d_o$ being the dimension of $O$. In conclusion,  if we use the HAL-MLE $\theta_N=\theta_{C_N,N}$ with cross-validation to select $C$ as the initial estimator of $\theta_0$ in the definition of the TMLE, then it will generally follow that  $\frac{1}{N}\sum_t R_{2,C_o(t)}(\theta_N^*,\theta_0)=o_P(N^{-1/2})$. This is an important result, as it is the main condition for the asymptotic normality and efficiency of the TMLE in Theorem \ref{thgen}.


\section{Context-specific causal effects of single-time point intervention}\label{sect3}

In this section we apply our general statistical formulation and TMLE (as described in great generality in Section 2) to a specific data structure $O(t)=(A(t),Y(t),W(t))$, context-specific model ${\cal M}(C_o(t))$, and target parameter $\Psi_{C_o(t)}$ common to causal inference literature. In particular, we define $\Psi_{C_o(t)}$ as the estimand identifying the average causal effect of a single time-point intervention $A(t)$ on the next outcome $Y(t)$ conditional on $C_o(t)$. We define ${\cal M}(C_o(t))$ as the nonparametric statistical model, relying only on possible knowledge on the conditional distribution of the treatment node $A(t)$ in $O(t)$. We proceed to define the efficient influence curve and exact second order expansion for this target parameter $\Psi_{C_o(t)}$, relying on the results from the i.i.d. literature. Subsequently, we proceed to establish the TMLE of the average over time of these context-specific causal effects, and apply our general Theorem 1 to analyze this TMLE.

\subsection{Statistical formulation}
{\bf Data:}
Let the observed data be $O(t)=(A(t),Y(t),W(t))$, $t=1,\ldots,N$, where $O(t)$ is of a fixed dimension in time $t$, and is an element of a Euclidean set ${\cal O}$. Let $A(t)\in \{0,1\}$ be a binary treatment, $Y(t)$ subsequent outcome that is either a binary outcome in $\{0,1\}$ or bounded continuous outcome in $(0,1)$. Additionally, we denote $W(t)$ as all other information collected after $A(t)$ that will be included in the history $C_o(t+1)$ for the next record $O(t+1)$, with history being defined as in Section 2.1. Finally, let $O^N=(O(t):t=1,\ldots,N)$ and let $P^N$ denote a possible probability measure. As before, we can factorize the probability density of the data according to the time ordering as follows:
\begin{align*}
p^N(o) &= \prod_{i=1}^{N}p_{a(t)}(a(t) | \bar{o}(t-1))  \prod_{i=1}^{N}p_{y(t)}(y(t) | \bar{o}(t-1),a(t)) \\ 
&\phantom{{}=0} \prod_{i=1}^{N}p_{w(t)}(w(t) | \bar{o}(t-1),y(t),a(t)) .
\end{align*}
Here, $p_{a(t)}$, $p_{y(t)}$ and $p_{w(t)}$ denote the conditional probability densities of $A(t)$, $Y(t)$ and $W(t)$ given the relevant past. We define $\mu_a$, $\mu_y$ and $\mu_w$ as the corresponding dominating measures. 

\vspace{2mm}
\noindent
{\bf Statistical model for time-series:}
We assume that $P_{O(t)\mid \bar{O}(t-1)}$ depends on $\bar{O}(t-1)$ through a summary measure $C_o(t)=C_o(\bar{O}(t-1))\in {\cal C}$ of fixed dimension. For notational convenience, we define the conditional distribution with $P_{C_o(t)}$. From the likelihood stated above, we can see that the density $p_{C_o(t)}(a(t),y(t),w(t)\mid C_o(t))$ factorizes into three conditional densities corresponding to $A(t)$, $Y(t)$, and $W(t)$, respectively. We denote these densities with $g_{a(t)}$, $q_{y(t)}$ and $q_{w(t)}$ to follow typical notation used in our previous work in the i.i.d. causal inference literature. We also define $C_a(t)=C_o(t)$, $C_y(t)=(C_o(t),A(t))$ and $C_w(t)=(C_o(t),A(t),Y(t))$ to be their corresponding fixed-dimensional relevant histories. As such, we assume that $g_{a(t)} = g_{a(t)}(a(t)\mid C_o(t))$ and $q_{y(t)} = q_{y(t)}(y(t)\mid C_y(t))$ are stationary in time, and we make no stationarity assumptions on $q_{w(t)}$. Since we impose conditional (strong) stationarity on $g_{a(t)}$ and $q_{y(t)}$, we have that $g_{a(t)}(a(t)\mid C_o(t))=\bar{g}(a(t)\mid C_o(t))$ and $q_{y(t)}(y(t)\mid C_o(t),a(t))=\bar{q}_y(y(t)\mid C_o(t),a(t))$ for common functions $\bar{g}$ and $\bar{q}_y$. Therefore, 
\[
p_{C_o(t)}(a,y,w)=\bar{g}(a\mid C_a(t))\bar{q}_y(y\mid C_y(t))q_{w(t)}(w\mid C_w(t)).
\]
We note that we make no model assumptions on $\bar{g}$, $\bar{q}_y$ and $q_{w(t)}$. We suppress dependence of the conditional density on $q_{w(t)}$ in future reference, as this factor plays no role
in estimation. In particular, neither $\Psi_{C_o(t)}(P_{C_o(t)})$ nor its canonical gradient depend on  $q_{w(t)}$, allowing us to act as if $q_{w(t)}$ is known. We define $\theta=(\bar{g},\bar{q}_y)$ and let $\Theta={\cal G}\times {\cal Q}$ be the cartesian product of the two nonparametric parameter spaces for $\bar{g}$ and $\bar{q}_y$. Let $p_{\theta,C_o(t)}$ and $p_{\theta}^N$ be the density for $O(t)$ given $C_o(t)$ and $O^N$, implied by $\theta=(\bar{g},\bar{q}_y)$.  This formulation defines a statistical model ${\cal M}^N$ for $P^N$. As in Section 2.1, we define a statistical model conditional on realized summary as ${\cal M}(C_o(t))=\{p_{C_o(t)}=\bar{g}_{C_o(t)}\bar{q}_{y,C_o(t)}q_{w(t),C_o(t)}: (\bar{g},\bar{q}_y)\in \Theta\}$ for $p_{C_o(t)}$ for a given $C_o(t)$.

\vspace{2mm}
\noindent
{\bf Target parameter:}
Below we define our first target parameter as the causal effect of $A(t)$ on subsequent outcome $Y(t)$, conditional on $C_o(t)$. In particular, for a given $C_o(t)$, we define a target parameter $\Psi_{C_o(t)} :{\cal M}(C_o(t))\rightarrow\openr$ given by:
\[
\Psi_{C_o(t)}(\bar{q}_y)=E(Y(t)\mid C_o(t),A(t)=1)=\int y \bar{q}_y(y\mid C_o(t),1)d\mu_y(y).
\]
Often the real parameter of interest is the causal difference, or the average treatment effect defined as:
$$E(Y(t)\mid C_o(t),A(t)=1)-E(Y(t)\mid C_o(t),A(t)=0).$$
We emphasize that our estimator can be immediately generalized to this contrast or to the bivariate parameter with these two components. Let $\bar{Q}_y(C_o(t),a)=E(Y(t)\mid A(t)=a,C_o(t))$ be the conditional mean of $Y(t)$, which is determined by $\bar{q}_y$.
The canonical gradient of $\Psi_{C_o(t)}:{\cal M}(C_o(t))\rightarrow\openr$ at $p_{\theta,C_o(t)}$ is given by:
\[
D^*(\theta)(C_o(t),O(t))=\frac{I(A(t)=1)}{\bar{g}(1\mid C_o(t))}(Y(t)-\bar{Q}_y(C_o(t),1)).
\]
Note that, for a given $C_o(t)$, this canonical gradient is a function of $O\in {\cal O}$ which has conditional mean zero w.r.t. $P_{\theta,C_o(t)}$.
Similarly to the discussion before, we reiterate the easy extension to the more interesting difference and its canonical gradient:
\[
\left (\frac{I(A(t)=1)}{\bar{g}(1\mid C_o(t))} - \frac{I(A(t)=0)}{\bar{g}(0\mid C_o(t))} \right) (Y(t) - \bar{Q}_y(C_o(t), a(t))).
\]
Note that here we can define $\Psi_{C_o(t)}(\theta)=\Psi_{C_o(t)}(\bar{q}_y)$. Further, we describe another interesting target parameter, defined as the average of $C_o(t)$-causal effects. In particular, we define $\Psi^N:{\cal M}^N\rightarrow\openr$ of the data distribution $P^N\in {\cal M}^N$, which is a function of $\theta$:
\begin{eqnarray*}
\Psi^N(P^N)&=&\bar{\Psi}(\bar{q}_y)\equiv \frac{1}{N}\sum_{t=1}^N \Psi_{C_o(t)}(\bar{q}_y)\\
&=&\frac{1}{N}\sum_{t=1}^N E(Y(t)\mid C_o(t),A(t)=1).
\end{eqnarray*}
Similarly for the difference we have:
\begin{eqnarray*}
\Psi^N(P^N)=\frac{1}{N}\sum_{t=1}^N [E(Y(t)\mid C_o(t),A(t)=1) - E(Y(t)\mid C_o(t),A(t)=0)].
\end{eqnarray*}
We emphasize that $\bar{\Psi}(\theta)=\bar{\Psi}(\bar{q}_y)$ is a data dependent target parameter since its value depends on the realized $C_o(t)$, $t=1,\ldots,N$.

\subsection{Defining the TMLE of the average of context-specific causal effects}
We follow the outline described in Section 2 for defining the TMLE. In particular, we define the appropriate loss function and parametric family of fluctuations of the initial estimator with fluctuation parameter $\epsilon$. Further, we specify the universal least favorable submodel. 

\vspace{2mm}
\noindent
In order to define the appropriate loss, we let $L(\bar{Q}_y)(C_o(t),O(t))$ be a loss function for $\bar{Q}_y$. In particular, we define $L(\bar{Q}_y)(C_o(t),O(t))$ as:
\[
-\{Y(t)\log \bar{Q}_y(C_o(t),A(t))+(1-Y(t))\log (1-\bar{Q}_y)(C_o(t),A(t))\}
\]
where $L(\bar{Q}_y)(C_o(t),O(t))$ is the log-likelihood loss for $\bar{Q}_y$. We emphasize that $P_{\theta_0,C_o(t)}L(\bar{Q}_{y,0})= \min_{\bar{Q}_y}P_{\theta_0,C_o(t)}L_{C_o(t)}(\bar{Q}_y)$, with $\bar{Q}_{y,0}$ being the truth.

\vspace{2mm}
\noindent
For a $\bar{Q}_{y}$ in our statistical model, we proceed to define a parametric working model $\{\bar{Q}_{y,\epsilon}:\epsilon\}$  through $\bar{Q}_y$  at $\epsilon =0$
with finite-dimensional parameter. We define the universal least favorable submodel with a logistic fluctuation: 
\[
 \mbox{Logit}\bar{Q}_{y,\epsilon}=\mbox{Logit}\bar{Q}_y+\epsilon H(\bar{g})
 \]
where $H(\bar{g})(C_o(t),A(t))=I(A(t)=1)/\bar{g}(A(t)\mid C_o(t))$ is the clever covariate for the $\Psi_{C_o(t)}(\bar{q}_y)$ target parameter, analogue to the i.i.d. TMLE of the treatment specific mean. Note that  for each $\epsilon$, we have that:
 \[
\frac{d}{d\epsilon}L(\bar{Q}_{y,\epsilon}) =D^*(\bar{Q}_{y,\epsilon},\bar{g})
 \]
 \noindent
For notational convenience, let $\theta_0=(\bar{Q}_{0,y},\bar{g}_0)$, which represents the only relevant part of $\theta$ the target parameter in question and its efficient influence curve depend on.
We define $\theta_N=(\bar{Q}_{y,N},\bar{g}_N)$ as the initial estimator of $\theta_0=(\bar{Q}_{0,y},\bar{g}_0)$. In particular, $\bar{Q}_{y,N}$ and $\bar{g}_N$ could be obtained by using the (e.g., online) Super Learner based on $\sum_t L(\bar{Q}_y)(C_o(t),O(t))$ and $\sum_t L_1(\bar{g})(C_o(t),O(t))$ loss, respectively, where $L_1(\bar{g})(C_o(t),A(t))=-\{A(t)\log\bar{g}(1\mid C_o(t))+(1-A(t))\log(1-\bar{g}(1\mid C_o(t)))\}$. 
For instance,  $\bar{Q}_{y,N}$ could be a Super-Learner estimate based on the $\sum_t L(\bar{Q}_y)(C_o(t),O(t))$ loss using the online cross-validation selector. 
Here we could include the HAL-MLE as a candidate estimator in the library of the online-super learner, beyond parametric model based MLEs and other machine learning algorithms. 
Given the initial estimator of $\theta_0=(\bar{Q}_{0,y},\bar{g}_0)$, we compute the maximum likelihood estimator of $\epsilon$ for the least favorable submodel through $\bar{Q}_{y,N}$ given by:
\[
\epsilon_N=\arg\min_{\epsilon}\sum_t L(\bar{Q}_{y,N,\epsilon})(C_o(t),O(t)).
\]
Let $\bar{Q}_{y,N}^*=\bar{Q}_{y,N,\epsilon_N}$ and $\theta_N^*=(\bar{Q}_{y,N}^*,\bar{g}_N)$ be the resulting update.
The score equation of this MLE yields:
\[
\sum_t D^*(\bar{Q}_{y,N}^*,\bar{g}_N)(C_o(t),O(t)) = 0.
\]

\subsection{Analysis of the TMLE}

In this subsection we analyze the TMLE by application of Theorem \ref{thgen}. Recall that we define the average of $C_o(t)$-causal effects as $\Psi^N:{\cal M}^N\rightarrow\openr$ of the data distribution $P^N\in {\cal M}^N$. In particular, this target parameter is a function of $\theta$ as: $$\bar{\Psi}(\bar{q}_y)\equiv \frac{1}{N}\sum_{t=1}^N \Psi_{C_o(t)}(\bar{q}_y) = \frac{1}{N}\sum_{t=1}^N E(Y(t)\mid C_o(t),A(t)=1).$$ 
First, we define necessary conditions for our Theorem \ref{thsingint}. We refer to Theorem \ref{thhal} showing that the HAL-MLE will indeed have the desired convergence w.r.t. the loss-based dissimilarity, as needed for the second order remainder and consistency conditions of this theorem.

\vspace{2mm}
\noindent
\begin{enumerate}
\item Define ${\cal F}$ to be a class of multivariate, real valued cadlag functions on an Euclidean cube $[0,\tau]$ containing ${\cal C}\times {\cal O}$ with sectional variation norm bounded by a universal constant $M<\infty$. We assume $\{D^*(\bar{Q}_y,\bar{g}):(\bar{Q}_y,\bar{g})\}\subset {\cal F}$, so that, in particular, $D^*(\theta_N^*)\in {\cal F}$ with probability 1. 
\item Assume $\frac{1}{N} \sum_t P_{\theta_0,C_o(t)}\{D^*(\bar{Q}_{y,N}^*,\bar{g}_N)-D^*(\bar{Q}_y^*,\bar{g}_0)\}^2\rightarrow_p 0$.
\item Assume negligible or asymptotic linearity of the remainder, such that:

\begin{eqnarray*}
\frac{1}{N}\sum_t \frac{\bar{g}_N-\bar{g}_0}{\bar{g}_N}(\bar{Q}_{y,N}^*-\bar{Q}_{y}^*)&=&o_P(N^{-1/2})\\
\frac{1}{N}\sum_t \frac{(\bar{g}_N-\bar{g}_0)^2}{\bar{g}_N\bar{g}_0}(\bar{Q}_y^*-\bar{Q}_{0,y})&=&o_P(N^{-1/2}).
\end{eqnarray*}

Additionally, assume that for some function $f$ we have that: 
$$\frac{1}{N} \sum_t \frac{\bar{g}_N-\bar{g}_0}{\bar{g}_0}(\bar{Q}_y^*-\bar{Q}_{0,y}) = \frac{1}{N} \sum_t f(C_o(t))(A(t)-\bar{g}_0(1\mid C_o(t)) + o_P(N^{-1/2}).$$

\item Let $\bar{f}=D^*(\bar{Q}_y^*,\bar{g}_0)+f(C_o(t))(A(t)-\bar{g}_0(1\mid C_o(t))$. Assume $\frac{1}{N}\sum_t P_{0,C_o(t)}\bar{f}^2\rightarrow\sigma^2_0$ a.s.
\end{enumerate}

\noindent
Note that if we assume that $\bar{Q}_{y,N}$ is consistent for $\bar{Q}_{0,y}$, then the three sub-conditions in C3 can be replaced by a single condition:
\[
\frac{1}{N}\sum_t \frac{\bar{g}_N-\bar{g}_0}{\bar{g}_N}(\bar{Q}_{y,N}^*-\bar{Q}_{y,0})=o_P(N^{-1/2}).
\]
We also note that if $\bar{g}_N$ is an MLE according to a parametric model, then the martingale approximation in C3 would be true under weak regularity conditions. 
\noindent
\begin{theorem}[Average over time Context-Specific Effect of a Single Intervention]\label{thsingint} \ \\ 
Let $\theta_N^*$ be the one-step TMLE satisfying $\sum_t D^*(\theta_N^*)(C_o(t),O(t))=0$, where $\theta_N^*=(\bar{Q}_{y,N}^*, \bar{g}_N)$. If C1, C2, C3 and C4 hold, then:
\vspace{3mm}
\[
\sqrt{N}(\bar{\Psi}(\theta_N^*)-\bar{\Psi}(\theta_0))\Rightarrow N(0,\sigma^2_0).
\] 

\vspace{3mm}
\noindent
where  $\sigma^2_0$ is the limit of $\frac{1}{N} \sum_t \bar{f}^2(C_o(t),O(t))$ and $\bar{f}=D^*(\theta^*)+f_1(\theta^*,\theta_0)$.
\end{theorem}

\vspace{3mm}
\begin{proof}
First, we have
\begin{equation}\label{fe1}
\begin{aligned}
\Psi_{C_o(t)}(\theta^*_N)-\Psi_{C_o(t)}(\theta_0) &=  - P_{\theta_0,C_o(t)}D^*_{C_o(t)}(\theta_N^*) \\
&\phantom{{}=0} + R_{2,C_o(t)}(\theta_N^*,\theta_0),
\end{aligned}
\end{equation}
where the second order remainder for any $\theta$ is given by:
\[
R_{2,C_o(t)}(\theta,\theta_0)=\frac{\bar{g}-\bar{g}_0}{\bar{g}}(1\mid C_o(t))(\bar{Q}_y-\bar{Q}_{y,0})(C_o(t),1).
\]
By combining the efficient score equation with the above second order expansion (\ref{fe1}) of $\Psi_{C_o(t)}$, we obtain the following expression for the Taylor expansion:

\[
\begin{array}{l}
\frac{1}{N} \sum_t [\Psi_{C_o(t)} (\theta^*_N) - \Psi_{C_o(t)}(\theta_0)] = \frac{1}{N} \sum_t \left\{ D^*(\bar{Q}_{y,N}^*, \bar{g}_N) (C_o(t),O(t)) - P_{\theta_0,C_o(t)} D^*(\bar{Q}_{y,N}^*, \bar{g}_N)\right\} \\
\hspace*{7cm} + \frac{1}{N}\sum_t R_{2,C_o(t)} (\bar{Q}_{y,N}^*, \bar{g}_N, \bar{Q}_{y,0}, \bar{g}_0).
\end{array}
\]

\vspace{2mm}
\noindent
We consider a martingale process $(M_N(f) : f \in {\cal F})$ defined by: 
\[
M_n(f)= \frac{1}{N} \sum_{t=1}^N \{f(C_o(t),O(t)) - P_{\theta_0,C_o(t)} f \}.
\]

\vspace{2mm}
\noindent
Define $\theta^*=(\bar{Q}^*_y, \bar{g}_0)$ as the limit of $\theta^*_N=(\bar{Q}^*_{y,N}, \bar{g}_N)$. By Lemma \ref{lemmaasympequi} and the fact that the class of functions ${\cal F}$ satisfies the entropy integral condition, we have the desired asymptotic equicontinuity of the martingale process so that by condition C2: 
\[
\begin{array}{l}
\frac{1}{N} \sum_t \left\{ D^*(\theta_N^*)(C_o(t),O(t))-P_{\theta_0,C_o(t)}D^*(\theta_N^*)\right\}\\
-\frac{1}{N}\sum_t\left\{ D^*(\theta^*)(C_o(t),O(t))-P_{\theta_0,C_o(t)}D^*(\theta^*)\right\}
=o_P(N^{-1/2}).
\end{array}
\]


\noindent 
We note that $R_{2,C_o(t)}(\bar{Q}_y^*,\bar{g}_0,\bar{Q}_{y,0},\bar{g}_0)=0$ for all $C_o(t)$. 
We now consider the remainder term, $\frac{1}{N}\sum_t R_{2,C_o(t)}(\theta_N^*,\theta_0)$, which can be represented as:
\begin{align*}
\frac{1}{N}\sum_t R_{2,C_o(t)}(\theta_N^*,\theta_0) &= \frac{1}{N}\sum_t \frac{\bar{g}_N-\bar{g}_0}{\bar{g}_N}(1\mid C_o(t))(\bar{Q}_{y,N}^*-\bar{Q}_{0,y})(C_o(t),1) \\
&= \frac{1}{N}\sum_t \frac{\bar{g}_N-\bar{g}_0}{\bar{g}_N}(\bar{Q}_{y,N}^*-\bar{Q}_{y}^*)+ \frac{1}{N}\sum_t \frac{\bar{g}_N-\bar{g}_0}{\bar{g}_N}(\bar{Q}_y^*-\bar{Q}_{0,y}).
\end{align*}

\noindent
By C3, we assume that the first term is a second order term so that $\bar{g}_N$ and $\bar{Q}_{y,N}$ converge fast enough to their limits ($o_P(N^{-1/2})$). We decompose the second term further, obtaining the following expression for $\frac{1}{N}\sum_t \frac{\bar{g}_N-\bar{g}_0}{\bar{g}_N}(\bar{Q}_y^*-\bar{Q}_{0,y})$:
\[
\frac{1}{N}\sum_t \frac{\bar{g}_N-\bar{g}_0}{\bar{g}_0}(\bar{Q}_y^*-\bar{Q}_{0,y})+
\frac{1}{N}\sum_t \frac{(\bar{g}_N-\bar{g}_0)^2}{\bar{g}_N\bar{g}_0}(\bar{Q}_y^*-\bar{Q}_{0,y}).
\]
By condition C3, the second term in the above expression is $o_P(N^{-1/2})$.
Condition C3 also assumes that the first term, a smooth function of   $\bar{g}_N-\bar{g}_0$, can be represented as $\frac{1}{N} \sum_t f(C_o(t))(A(t)-\bar{g}_0(1\mid C_o(t))$ for some $f$, plus $o_P(N^{-1/2})$, so that it is a martingale. It follows that:
\begin{eqnarray*}
\bar{\Psi}(\bar{Q}_{y,N}^*)-\bar{\Psi}(\bar{Q}_{0,y}) &=&
\frac{1}{N}\sum_t\left\{ D^*(\bar{Q}_y^*,\bar{g}_0)(C_o(t),O(t))-P_{\theta_0,C_o(t)}D^*(\bar{Q}_y^*,\bar{g}_0) \right\}\\
&& + \frac{1}{N}\sum_t f(C_o(t))(A(t)-\bar{g}_0(1\mid C_o(t)) +o_P(N^{-1/2}).
\end{eqnarray*}

\noindent
The right-hand side is a discrete martingale up until an $o_P(N^{-1/2})$ term, whose standardized version converges to a normal limit distribution. As a consequence, by condition C4, $N^{1/2}(\bar{\Psi}(\theta_N^*)-\bar{\Psi}(\theta_0))\Rightarrow N(0,\sigma^2_0)$ which proves our result.

\end{proof}


\section{Context-specific causal effects of multiple time-point interventions}\label{sect4}

As opposed to the setting described in the previous section, one might be interested in interventions over multiple consecutive time points, analogous to well-studied longitudinal settings \cite{gruber2011}. In this section, we consider a more general longitudinal data structure $O(t)$ involving multiple intervention nodes $A(t,k)$, $k=0,\ldots,K$, alternated with time-dependent covariate nodes $L(t,k)$, $k=1,\ldots,K$, and a final outcome $Y(t)=L(t,K+1)$. Conditional on the context $C_o(t)$ at time $t$, we define $\Psi_{C_o(t)}(P^N)$ as the counterfactual mean outcome  of $Y(t)$ under a multiple time-point intervention on these $t$-specific intervention nodes, generalizing the results obtained in the previous section. As before, we define the statistical model ${\cal M}_{C_o(t)}$ for the conditional distribution of $O(t)$, given $C_o(t)$, analogue to the i.i.d. literature.  As before, out target parameter of interest is the average over time of context-specific counterfactual mean outcomes. We present the targeted maximum likelihood estimator involving estimation of the conditional density of $O(t)$ given $C_o(t)$, and the sequential regression based TMLE, analogue to the i.i.d. literature. Finally, we apply our general Theorem \ref{thgen} to these two TMLEs, resulting in two new theorems. 
\subsection{Statistical formulation}
{\bf Data:}
We define the observed data as: $$O(t)=(A(t,0),L(t,1),A(t,1),\ldots,L(t,K),A(t,K),L(t,K+1)), \ \ \ t=1,\ldots,N,$$ where $O(t)$ is a fixed dimensional element of an Euclidean set ${\cal O}$. In particular, $O(t)$ is an ordered longitudinal data structure within time unit $t$, with an intervention node $A(t,j)$ representing treatment or censoring. For notational convenience, we define $A(t)$ and $L(t)$ as $A(t) = (A(t,0),  \cdots, A(t,K))$ and $L(t) = (L(t,1), \cdots, L(t,K))$. Note that $L(t,j+1)$ is a vector of subsequent time-dependent covariates at time $j$ within $t$, $t=1,\ldots,N$. Let $Y(t) $ be a component or real valued function of $L(t,K+1)$, where $Y(t)$ is  the outcome of interest. We note that with this formulation, the complete time-series we observe, $O^N \sim P_0^N$, is just an alternation of time-dependent treatment and time-dependent covariate/outcomes. We emphasize that the blocks $O(t)$ could have a relevant interpretation in line with the applied problem at hand. As such, the measurements in $O(t)$ might correspond with a sequence of unique actions and measurements on day/cycle/period $t$, so that only $A(t,j)$ and $L(t,j)$ across $t$ for a fixed $j$ are measuring the same $j$-specific variable at time $t$. Similarly, it could be the case that we truly observe a unique experiment over a time block $t$ - for example, at each time $t$ a new subject/unit enrolls and is observed over time points $j$. When no interpretable block is possible, we emphasize that creating artificial blocks $O(t)$ solely for the purpose of estimation of a particular causal effect is also in line with our developed theory. The probability density of $P^N$ of $O^N$ can be factorized according to the time-ordering as follows:
\begin{align*}
p^N(o) &= \prod_{t=1}^{N}p_{a(t)}(a(t) | \bar{o}(t-1))  \prod_{t=1}^{N} p_{l(t)}( l(t) | \bar{o}(t-1),a(t)) \\ 
&= \prod_{t=1}^{N}  \prod_{j=0}^{K}  p_{a(t,j)}(a(t,j) | \bar{o}(t-1,j))  \prod_{t=1}^{N} \prod_{j=1}^{K+1} p_{l(t,j)}( l(t,j) | \bar{o}(t-1,j),a(t,j)).
\end{align*}
Here, $p_{a(t,j)}$ and $p_{l(t,j)}$ denote the conditional probability densities of $A(t,j)$ and $L(t,j)$ given the relevant past. We define $\mu_a$ and $\mu_l$ as the corresponding dominating measures. 

\vspace{2mm}
\noindent
{\bf Statistical model:}
As in the previous section, we assume the conditional density of $O(t)$ given $\bar{O}(t-1)$, $P_{O(t)\mid \bar{O}(t-1)}$, depends on the past only through a summary measure $C_o(t)=C_o(\bar{O}(t-1))\in {\cal C}$ of fixed dimension. We denote $P_{O(t) \mid C_o(t) }$ with $P_{C_o(t)}$ to simplify notation. Intuitively, at time $t$, $C_o(t)$ plays the role of baseline-covariates for the $t^{\text{th}}$ experiment $P_{O(t)\mid \bar{O}(t-1)}$ of the ordered sequence of experiments $t=1,\ldots,N$. Further, let $g_{t,j} = p_{a(t,j)}$ and $q_{t,j} = p_{l(t,j)}$. With the new notation set, the density $p_{C_o(t)}=p_{O(t)\mid C_o(t)}$ can be factorized as follows:
\[
p_{C_o(t)}(O(t))=\prod_{j=1}^{K+1}q_{t,j}(L(t,j)\mid C_l(t,j))\prod_{j=0}^K g_{t,j}(A(t,j)\mid C_a(t,j)),
\]
where $C_l(t,j)=(L(t,1:j-1),A(t,1:j-1),C_o(t))$ and $C_a(t,j) = (L(t,1:j),A(t,1:j-1),C_o(t))$ denote the relevant histories for the conditional densities of $A(t,j)$ and $L(t,j)$ depend on. 

\vspace{2mm}
\noindent
We assume that the time-series is described by a common-in-time $\bar{q}_j$ for $j=1,\ldots,K+1$. Further, we partition the indices $\{1,\ldots,K\}$ for the intervention nodes into two disjoint and complementary sets ${\cal A}_1$ and ${\cal A}_2$. We assume that for $j\in {\cal A}_1$, the intervention mechanism $g_{t,j}$ for generating $A(t,j)$ is known for each $t$. Since the treatment mechanism $g_{t,j}$ for $j\in {\cal A}_1$ is known, we do not need conditional stationarity assumptions for $j\in {\cal A}_1$. We make this distinction in order to be able to study settings when treatment is sequentially randomized with probabilities that change across time $t$. For $j\in {\cal A}_2$, we assume that $g_{t,j}$ is described by a common-in-time $\bar{g}_j$. The density $p_{c_o(t)}$ is now modeled as follows:
\[
p_{c_o(t),\bar{q},\bar{g}}(o(t))\equiv \prod_{j=1}^{K+1}\bar{q}_j(l(t,j)\mid c_l(t,j))\prod_{j\in {\cal A}_1}g_{t,j}(a(t,j)\mid c_a(t,j))\prod_{j\in {\cal A}_2}\bar{g}_j(a(t,j)\mid c_a(t,j)).
\]
Let $\bar{q}=(\bar{q}_1,\ldots,\bar{q}_{K+1})$ and $\bar{g}=(\bar{g}_j:j\in {\cal A}_2)$ be the unknown parameters in this representation of the density of $O(t)$ given $\bar{O}(t-1)$. Define ${\cal Q}_j$ as the  nonparametric set of conditional densities so that  $\bar{q}_j \in {\cal Q}_j$. We emphasize that we impose no restrictions on $\bar{q}_j$ in ${\cal Q}_j$. Similarly, let ${\cal G}_j$ be a possibly restricted set of conditional densities of $A(j)$, $j\in {\cal A}_2$. We define $\theta_j=(\bar{g}_j,\bar{q}_j)$ and let $\Theta_j={\cal G}_j \times {\cal Q}_j$ be the cartesian product of the parameter spaces for $\bar{g}_j$ and $\bar{q}_j$. With that, we have defined a statistical model ${\cal M}^N$ for $P^N$. Additionally, we defined a statistical model for $p_{C_o(t)}$ conditional on the realized fixed-dimensional summary $C_o(t)$ as the following ${\cal M}(C_o(t))$:
\[
{\cal M}(C_o(t))=\left\{ p_{C_o(t),\bar{q},\bar{g}}: \mbox{ $\forall$ }j,  \bar{q}_j\in {\cal Q}_j,\mbox{ $\forall$ } j\in {\cal A}_2, \bar{g}_j \in {\cal G}_j \right\}.
\]

\vspace{2mm}
\noindent
{\bf Target parameter:}
Let $\bar{g}^*_t=(\bar{g}_{t,j}^*:j=1,\ldots,K)$ denote conditional densities of $A(t,j)$  given a summary measure $C_a^*(t,j)$ of $C_o(t)$ and $C_a(t,j)$. We define the $G$-computation formula for the post-intervention distribution of $O(t)$ given $C_o(t)$ for $j=1,\ldots,K$ with $A(t,j)$ subjected to the intervention $\bar{g}^*=(\bar{g}_{t,j}^*:(t,j))$ across $t$ and $j$:
\[
p_{c_o(t),\bar{q},\bar{g}^*}(o(t))=\prod_{j=1}^{K+1}\bar{q}_j(l(t,j)\mid c_l(t,j))\prod_{j=1}^K \bar{g}_{t,j}^*(a(t,j)\mid c_a^*(t,j)).
\]
Typically, $\bar{g}^*_{t,j}=\bar{g}^*_j$ is constant over time $t$.
However, one is able to define $C_o(t)$-specific stochastic interventions at each  $C_o(t)$, $t=1,\ldots,N$, allowing for changing interventions across time blocks.  

\vspace{2mm}
\noindent
Recall that $Y(t)$ is the outcome of interest for the $C_o(t)$-specific experiment, and it is defined as a real valued function of $L(t,K+1)$. We emphasize that our analysis is flexible enough to support many different outcomes within the user-specified time block $t$. Let $Y_{\bar{g}^*}(t)$ be the random variable of $Y(t)$ with conditional probability density, given $C_o(t)$, implied by $p_{C_o(t),\bar{q},\bar{g}^*}$. In particular, we might be interested in the $C_o(t)$-specific counterfactual mean under stochastic intervention $\bar{g}^*$:
\[
\Psi_{C_o(t)}(\bar{q})=\Psi_{C_o(t)}(P_{C_o(t),\bar{q},\bar{g}})=E_{P_{C_o(t),\bar{q},\bar{g}^*}} Y_{\bar{g}^*}(t).
\]
Under a $C_o(t)$-specific structural equation model and sequential randomization of $A(t,j)$, $j=1,\ldots,K$, we note that $\Psi_{C_o(t)}(P_{C_o(t)})$ denotes the counterfactual mean outcome of $Y(t)$ under the stochastic intervention $\bar{g}^*$ given $C_o(t)$.

\vspace{2mm}
\noindent
Let $D^*_{C_o(t)}(\bar{q},\bar{g})$ be the canonical gradient of $\Psi_{C_o(t)}$ at $P_{C_o(t),\bar{q},\bar{g}}$ for the statistical model conditional on the realized fixed dimensional summary,  ${\cal M}(C_o(t))$. This canonical gradient is well known from the i.i.d. literature (e.g., \citep{bang&robins05}) and given by:
\[
D^*_{C_o(t)}(\bar{q},\bar{g})(O(t)) = \sum_{j=1}^{K+1} H_j(\bar{g}) (C_o(t),O(t)) (\bar{Q}_{A(t,j)}-\bar{Q}_{L(t,j)}) (C_o(t),O(t)),
\]
where  we will now elaborate on the relevant parts of this efficient influence curve, $D^*_{C_o(t)}(\bar{q},\bar{g})(O(t))$. First, we define the clever covariate as the product ratio of the treatment mechanism under the stochastic intervention and true intervention mechanism:
$$H_{j}(\bar{g})(C_o(t),O(t))=\frac{\prod_{l=1}^j\bar{g}^*_{t,l}(A(t,l)\mid C_a^*(t,l))}{\prod_{l=1}^jg_{t,l}(A(t,l)\mid C_a^*(t,l))}.$$
The conditional means $\bar{Q}_{A(t,j)}$ and $\bar{Q}_{L(t,j)}$ are defined recursively as follows. In particular, we define the conditional expectations for the first two iterations as:
\begin{eqnarray*}
\bar{Q}_{A(t,K+1)}&=& Y(t);\\
\bar{Q}_{L(t,K+1)}&=&E_{\bar{q}_{K+1}}(Y(t)\mid A(t,1:K),L(t,1:K),C_o(t));\\
\bar{Q}_{A(t,K)}&=&E_{\bar{g}_K^*}(\bar{Q}_{L(t,K+1)}\mid A(t,1:K-1),L(t,1:K),C_o(t));\\
\bar{Q}_{L(t,K)}&=&E_{\bar{q}_K}(\bar{Q}_{A(t,K)}\mid A(t,1:K-1),L(t,1:K-1),C_o(t)).
\end{eqnarray*}
Subsequently, we have that for $k=K-1,\ldots,0$:
\begin{eqnarray*}
\bar{Q}_{A(t,k)}&=&E_{\bar{g}_k^*}(\bar{Q}_{L(t,k+1)}\mid A(t,1:k-1),L(t,1:k),C_o(t)),
\end{eqnarray*}
and for $k=K-1,\ldots,1$:
\begin{eqnarray*}
\bar{Q}_{L(t,k)}&=& E_{\bar{q}_k}(\bar{Q}_{A(t,k)}\mid A(t:1:k-1),L(t,1:k-1),C_o(t)).
\end{eqnarray*}
For a specific $j$, we define $D^*_{C_o(t),j}(\bar{q},\bar{g})(O(t))$ for each $j=1,\ldots,K+1$:
\[
D^*_{C_o(t),j}(\bar{q},\bar{g})(O(t))=H_{j}(\bar{g})(C_o(t),O(t))(\bar{Q}_{A(t,j)}-\bar{Q}_{L(t,j)}),
\]
 so that $D^*_{C_o(t)}(\bar{q},\bar{g})=\sum_{j=1}^{K+1}D^*_{C_o(t),j}(\bar{q},\bar{g})$. The $C_o(t)$-specific counterfactual mean equals the last conditional expectation:
\begin{equation}\label{seqregrepr}
\Psi_{C_o(t)}(\bar{q})=\bar{Q}_{A(t,0)}(C_o(t)),
\end{equation}
which is often referred to as the sequential regression representation of the counterfactual mean outcome.

\vspace{2mm}
\noindent
Let
\[
\Psi(P^N)=\bar{\Psi}(\bar{q})=\frac{1}{N}\sum_{t=1}^N \Psi_{C_o(t)}(P_{C_o(t),\bar{q},\bar{g}^*})
\]
be the average over time of the context-specific counterfactual mean outcomes.
We also define the second order remainder: 
\[
R_{20,C_o(t)}(\bar{q},\bar{g},\bar{q}_0,\bar{g}_0)\equiv \Psi_{C_o(t)}(\bar{q})-\Psi_{C_o(t)}(\bar{q}_0)+P_{\theta_0,C_o(t)}D^*_{C_o(t)}(\bar{q},\bar{g}).\]
This term can be explicitly presented as a sum of terms involving integrals over cross-products of differences $(\bar{q}-\bar{q}_0)$ and $(\bar{g}-\bar{g}_0)$, and thereby has the double robustness structure. 

\subsection{TMLE involving estimation of the conditional density of $O(t)$ given $C_o(t)$}

In this subsection, we define the TMLE as described in Section 2, by defining the appropriate loss function and parametric fluctuation of the initial estimator.  First, we construct an initial estimator $(\bar{q}_N,\bar{g}_N)$ of $(\bar{q}_0,\bar{g}_0)$. While there are many ways to construct an initial estimator, we advocate for online Super-learner based estimators constructed based on the log-likelihood loss and appropriate cross-validation for time-series (recursive, hybrid, or rolling cross-validation scheme). Further, we need to define an appropriate loss function, namely $L(\bar{q})(C_o(t),O(t))$, as a loss function for $\bar{q}$ such that $P_{\theta_0,C_o(t)}L(\bar{q}_{0})(C_o(t),O(t)) = \min_{\bar{q}}P_{\theta_0,C_o(t)}L_{C_o(t)}(\bar{q})$. A natural loss function for $\bar{q}$ is the log-likelihood loss $-\sum_t \log \bar{q}(y(t)\mid c_y(t))$.  Similarly, we can use the log-likelihood loss $-\sum_t \log \bar{g}(a(t)\mid c_a(t))$ for $\bar{g}$. For a $\bar{q}_N$ in our statistical model, we proceed to define a parametric working model $\{\bar{q}_{N,\epsilon}:\epsilon\}$ with finite-dimensional parameter $\epsilon$ so that $\epsilon = 0$ indexed  $\bar{q}_N$. We define the universal least favorable submodel such that:
\[
\frac{d}{d\epsilon}\log p_{C_o(t),\bar{q}_{N,\epsilon},\bar{g}_N}=D^*_{C_o(t)}(\bar{q}_{N,\epsilon},\bar{g}_N).
\]
Similarly, for a local least favorable model, we would have:
\[
\left . \frac{d}{d\epsilon}\log p_{C_o(t),\bar{q}_{N,\epsilon},\bar{g}_N} \right |_{\epsilon =0} = D^*_{C_o(t)}(\bar{q}_{N},\bar{g}_N).
\]
A possible local least favorable submodel is the following parametric fluctuation model:
\[
\bar{q}_{N,j,\epsilon}=\bar{q}_{N,j}(1+\epsilon D^*_{C_o(t),j}(\bar{q}_N,\bar{g}_N)),
\]
with a common $\epsilon$, $j=1,\ldots,K$. A   local least favorable submodel implies a universal least favorable model  by tracking this local model iteratively for small local moves. We note that $\bar{q}_{N,j,\epsilon}$ is still a proper density within our statistical model for $D^*_{C_o(t),j}(\bar{q}_N,\bar{g}_N)$ uniformly bounded. Given the initial estimator $(\bar{q}_N,\bar{g}_N)$ of $(\bar{q}_0,\bar{g}_0)$, we compute the maximum likelihood estimator of $\epsilon$ given by:
\[
\epsilon_N=\arg\max_{\epsilon}\frac{1}{N}\sum_{t=1}^N \log p_{C_o(t),\bar{q}_{N,\epsilon},\bar{g}_N}(O(t)\mid C_o(t)).
\]
For the universal least favorable submodel, the score equation of the MLE yields:
\[
\frac{1}{N}\sum_{t=1}^N D^*_{C_o(t)}(\bar{q}_{N,\epsilon_N},\bar{g}_N)(O(t)) = 0.
\]
More generally, let $\bar{q}_N^*$ be an update of $\bar{q}_N$ so that the above equation holds up until $o_P(N^{-1/2})$, with $\bar{q}_N^*$ representing the final update for either the one-step or iterative final updated estimate based on a local least favorable submodel. We note that we can also include other TMLEs based on this local least favorable submodel- for example, we could include the closed form TMLE as described in \cite{vdl2010} and \cite{stitelman2011}. We note that the latter TMLE involves framing data in terms of binaries, using clever covariates for each binary conditional distribution with separate $\epsilon$. One then recursively carries out separate TMLE steps starting at the last factor of the ordered likelihood, and proceeding downwards until the first factor is targeted, always using the most recent targeted updates in the clever covariate. This i.i.d. TMLE can be applied to our data set of $t$-specific data structures $(C_o(t),O(t))$, where $C_o(t)$ represents the baseline covariate. The corresponding TMLE of $\bar{\Psi}(\bar{q}_0)$ is  given by $\bar{\Psi}(\bar{q}_N^*)$.

\vspace{2mm}
\noindent
We now apply Theorem \ref{thgen} to the average over time of context-specific causal effects with multiple time-point interventions, with the TMLE constructed through the estimation of conditional density of $O(t)$ given $C_o(t)$. In particular, recall that we define $C_o(t)$-specific counterfactual mean under stochastic intervention $\bar{g}^*$ as:
$$\Psi_{C_o(t)}(\bar{q})=\Psi_{C_o(t)}(P_{C_o(t),\bar{q},\bar{g}})=E_{P_{C_o(t),\bar{q},\bar{g}^*}} Y_{\bar{g}^*}(t)$$ where $\Psi^N:{\cal M}^N\rightarrow\openr$ of the data distribution $P^N\in {\cal M}^N$. Similarly, the average over time $C_o(t)$-specific target parameter was defined as: $$\Psi(P^N)=\bar{\Psi}(\bar{q})=\frac{1}{N}\sum_{t=1}^N \Psi_{C_o(t)}(P_{C_o(t),\bar{q},\bar{g}^*}).$$ 

\noindent
We first define the necessary conditions:

\vspace{2mm}
\noindent
\begin{enumerate}
\item Define ${\cal F}$ to be a class of multivariate, real valued cadlag functions on an Euclidean cube $[0,\tau]$ containing ${\cal C}\times {\cal O}$ with sectional variation norm $\pl f\pl_v^*$, $f \in \cal F$, bounded by an universal constant $M<\infty$. We assume ${\cal F}\equiv \{(C_o,o)\rightarrow D^*_{C_o}(\bar{q},\bar{g})(o):\bar{q}\in {\cal Q}, \bar{g}\in {\cal G}\}$ and $D^*_{C_o}(\bar{q}_N^*,\bar{g}_N)(o) \in {\cal F}$ with probability tending to 1. 

\item Assume $\frac{1}{N} \sum_{t=1}^N P_{0,C_o(t)} \{ D^*_{C_0}(\bar{q}_N^*,\bar{g}_N) - D^*_{C_0}(\bar{q}^*,\bar{g}_0)\}^2 \rightarrow_p 0$ as $N\rightarrow\infty$ for some possibly misspecified limit $\bar{q}^*$.

\item Assume negligible or asymptotic linearity of the remainder, such that:
\[
\frac{1}{N} \sum_{t=1}^N R_{2,C_o(t)}(\bar{q}_N^*,\bar{g}_N,\bar{q}^*,\bar{g}_0)=o_P(N^{-1/2}),
\] 
and
\[
\frac{1}{N} \sum_{t=1}^N R_{2,C_o(t)}(\bar{q}^*,\bar{g}_N,\bar{q}_0,\bar{g}_0)=M_N(f_1)+o_P(N^{-1/2})
\]
for some $f_1$, where $M_N(f)=\frac{1}{N} \sum_{t=1}^N \{f(C_o(t),O(t))-P_{0,C_o(t)}f\}$. Note however, that if $\bar{q}^*=\bar{q}_0$, we do not need the second part of the assumption since: $$\frac{1}{N} \sum_{t=1}^N R_{2,C_o(t)}(\bar{q}^*,\bar{g}_N,\bar{q}_0,\bar{g}_0) = 0.$$

\item Let $\bar{f}=D^*_{C_0}(\bar{q}^*,\bar{g}_0)+f_1$. Assume $\frac{1}{N} \sum_{t=1}^N P_{0,C_o(t)}\bar{f}^2\rightarrow\sigma^2_0$ a.s.
\end{enumerate}

\begin{theorem} [Conditional density based TMLE]\label{thmultgint}
Let $\bar{q}_N^*$ be the one-step or iterative targeted estimate of $\bar{q}_0$, such that: $$\frac{1}{N} \sum_{t=1}^N D^*_{C_o(t)}(\bar{q}_{N,\epsilon_N},\bar{g}_N)(O(t)) =o_P(N^{-1/2}).$$ If C1, C2, C3 and C4 hold, then:
\vspace{3mm}
\[
\sqrt{N} (\bar{\Psi}(\bar{q}_N^*) - \bar{\Psi}(\bar{q}_0)) \Rightarrow N(0,\sigma^2),
\] 
\vspace{5mm}
\noindent
where  $\sigma^2_0$ is the limit of $\frac{1}{N} \sum_{t=1}^N P_{0,C_o(t)}\bar{f}^2$ and $\bar{f}=D^*_{C_0}(\bar{q}^*,\bar{g}_0)+f_1$.
\end{theorem}

\begin{proof} The second order expansion for the $C_o(t)$-specific multiple-time-point intervention target parameter is as follows: \begin{equation}\label{fe1}\begin{aligned}\Psi_{C_o(t)}(\bar{q}^*_N, \bar{g}_N)-\Psi_{C_o(t)}(\bar{q}_0, \bar{g}_0) &=  - P_{0,C_o(t)} D^*_{C_o(t)}(\bar{q}^*_N, \bar{g}_N)\} \\&\phantom{{}=0} + R_{2,C_o(t)}(\bar{q}^*_N, \bar{g}_N,\bar{q}_0,\bar{g}_0),\end{aligned}\end{equation}where $R_{2,C_o(t)}(\bar{q}^*_N, \bar{g}_N,\bar{q}_0,\bar{g}_0)$ defined above represents a double robust structured remainder, a difference between $(\bar{q}^*_N, \bar{g}_N)$ and $(\bar{q}_0,\bar{g}_0)$, and $\Psi_{C_o(t)}(\bar{q}_N^*, \bar{g}_N)$ is the TMLE. By combining this equation with the efficient score equation, we obtain the following second order expansion for the average over time $C_o(t)$-specific multiple-time-point intervention target parameter:\begin{eqnarray*}\bar{\Psi}(\bar{q}^*_N, \bar{g}_N) - \bar{\Psi}(\bar{q}_0, \bar{g}_0)& =& \frac{1}{N} \sum_{t=1}^N \{D^*_{C_o(t)}(\bar{q}^*_N,\bar{g}_N)-P_{0,C_o(t)} D^*_{C_o(t)}(\bar{q}^*_N, \bar{g}_N)\} \\&&+ \frac{1}{N} \sum_{t=1}^N R_{2,C_o(t)}(\bar{q}^*_N, \bar{g}_N,\bar{q}_0, \bar{g}_0).\end{eqnarray*}

\noindent
We note that ${\cal F}$ satisfies the required entropy conditions by C1. Consider a martingale process $(M_N(f) : f \in {\cal F})$ defined by: \[M_N(f)= \frac{1}{N} \sum_{t=1}^N \{f(C_o(t),O(t)) - P_{0,C_o(t)} f \}\]By Lemma \ref{lemmaasympequi}, we note that $M_N(f)$ is a martingale process that is asymptotically equicontinuous in the sense that convergence $M_N(f_N)-M_N(f)=o_P(N^{-1/2})$ if $\frac{1}{N} \sum_{t=1}^N P_{0,C_o(t)}\{f_N-f\}^2 \rightarrow_p 0$. Define $\theta^*=(\bar{q}^*, \bar{g}_0)$ as the limit of $\theta^*_N=(\bar{q}^*_{N}, \bar{g}_N)$. By C2 and the definition of asymptotic equicontinuity of a martingale process, we have that: \[\begin{array}{l}\frac{1}{N} \sum_{t=1}^N \left\{ D^*_{C_0}(\theta_N^*)(C_o(t),O(t)) - P_{0,C_o(t)}D^*_{C_0}(\theta_N^*)\right\}\\-\frac{1}{N}\sum_{t=1}^N \left\{ D^*_{C_0}(\theta^*)(C_o(t),O(t)) - P_{0,C_o(t)}D^*_{C_0}(\theta^*)\right\}=o_P(N^{-1/2}).\end{array}\]
The consistency condition of Theorem \ref{thgen} is established by noting that for all possible $C_o(t)$, we have that $R_{2,C_o(t)}(\bar{q}^*,\bar{g}_0,\bar{q}_0,\bar{g}_0)=0$. The second order expansion now takes the following form:\[\begin{array}{l}\bar{\Psi}(\bar{q}^*_N, \bar{g}_N) - \bar{\Psi}(\bar{q}_0, \bar{g}_0)=  \frac{1}{N} \sum_{t=1}^N \left\{ D^*(\bar{q}^*, \bar{g}_0) (C_o(t),O(t)) - P_{0,C_o(t)} D^*(\bar{q}^*, \bar{g}_0) \right\} \\\hspace*{6cm} + o_p(N^{-1/2}) +  \frac{1}{N} \sum_{t=1}^N R_{2,C_o(t)} (\bar{q}_{N}^*, \bar{g}_N, \bar{q}_{0}, \bar{g}_0).\end{array}\]

\vspace{2mm}
\noindent
Consider now the term:\begin{equation*}\label{ltmlea}\frac{1}{N}\sum_{t=1}^N R_{2,C_o(t)}(\bar{q}_N^*,\bar{g}_N,\bar{q}_0,\bar{g}_0).\end{equation*}Due to the double robust structure of the second order remainder, we can represent $R_{2,C_o(t)}(\bar{q},\bar{g},\bar{q}_0,\bar{g}_0)$ as a sum of terms with structure $\int (H_1(\bar{g})-H_1(\bar{g}_0))(H_2(\bar{q})-H_2(\bar{q}_0))H_3(P_0)dP_0$ for some specified $H_1,H_2$ and $H_3$. Therefore we have the following decomposition of the second order remainder, which considers convergence of $\bar{q}$ and $\bar{g}$ separately:\[R_{2,C_o(t)}(\bar{q}_N^*,\bar{g}_N,\bar{q}_0,\bar{g}_0)=R_{2,C_o(t)}(\bar{q}_N^*,\bar{g}_N,\bar{q}^*,\bar{g}_0)+R_{2,C_o(t)}(\bar{q}^*,\bar{g}_N,\bar{q}_0,\bar{g}_0).\] 
By the Cauchy-Schwarz inequality, we can bound the first term by the product of $L^2(P_0)$-norm of $H_1(\bar{g}_N) - H_1(\bar{g}_0)$ and $L^2(P_0)$-norm of $H_2(\bar{q}_N^*)-H_2(\bar{q}^*)$. Therefore, it is reasonable to assume  assumption C3, so that $\frac{1}{N} \sum_{t=1}^N R_{2,C_o(t)} (\bar{q}_N^*,\bar{g}_N,\bar{q}^*,\bar{g}_0)=o_P(N^{-1/2})$.
Regarding the second term, we consider two cases.
\vspace{2mm}
\noindent
\textbf{Case 1: $\bar{q}^*=\bar{q}_0$:} Note that the second term in the above expression is zero for $\bar{q}^*=\bar{q}_0$. 

\vspace{2mm}
\noindent
\textbf{Case 2: $\bar{q}^*\not =\bar{q}_0$:} The second order remainder, $\frac{1}{N} \sum_{t=1}^N R_{2,C_o(t)}(\bar{q}^*,\bar{g}_N,\bar{q}^*,\bar{g}_0)$ corresponds with a term of the following structure for a certain $f$: $$\int (H_1(\bar{g}_N)-H_1(\bar{g}_0)) f dP_0.$$ Condition C3 assumes that $\int (H_1(\bar{g}_N)-H_1(\bar{g}_0)) f dP_0$ can be approximated by a martingale sum $M_N(f_1)$ for some $f_1$. If $\bar{g}_N$ is an MLE according to a correct model, then this assumption will follow naturally since MLEs can be approximated by martingale sums, just as MLE for i.i.d. parametric models. On the other hand if $\bar{g}_N$ is estimated using machine learning, we note that this condition could be potentially problematic. A possible solution would be to target $\bar{g}_N$ towards the particular parameter $\bar{g} \rightarrow \int H_1(\bar{g}) fdP_0$. We refer the interested reader to \cite{benkeser2017,vdl2014dr} for more details regarding this targeting strategy. Under the above assumptions, we have established the following:\begin{eqnarray*}\bar{\Psi}(\bar{q}_N^*)-\bar{\Psi}(\bar{q})&=&\frac{1}{N} \sum_{t=1}^N  \left\{ D^*_{C_o(t)}(\bar{q}^*,\bar{g}_0)(O(t)) - \bar{P}_{0,C_o(t)}D^*_{C_o(t)}(\bar{q}^*,\bar{g}_0) \right\}\\&&+\frac{1}{N} \sum_{t=1}^N \left\{f_1(C_o(t),O(t))-P_{0,C_o(t)}f_1\right\} + o_P(N^{-1/2}).\end{eqnarray*}
 As a consequence, we have that:\[\sqrt{N} (\bar{\Psi}(\bar{q}_N^*) - \bar{\Psi}(\bar{q}_0)) \Rightarrow N(0,\sigma^2).\]
 We note that the second term will only reduce the asymptotic variance, and as such one could decide to ignore it in variance estimation. The resulting inference is asymptotically conservative, except if $\bar{q}^*=\bar{q}_0$. 
 \end{proof}

\subsection{Defining the TMLE through sequential regression}
The TMLE described in the previous two subsections requires the estimation of the conditional density $p_{0,C_o(t)}$ of $O(t)$ given $C_o(t)$. More specifically, in order to construct a TMLE we need to estimate conditional densities $(\bar{q}_0,\bar{g}_0)$. This could be potentially problematic, since if $L(t,j)$ is high-dimensional, the construction of initial estimators of $\bar{q}_{0,j}$ is challenging. In order to alleviate this issue, we remind that the target parameter $\Psi_{C_o(t)}(\bar{q})$ could also be written as an iterative conditional expectation, as noted in (\ref{seqregrepr}). In particular, we have that:
\[
\Psi_{C_o(t)}(\bar{q})=\bar{Q}_{A(t,0)}(C_o(t)),
\]
which provides an iterative conditional expectation representation of $\Psi_{C_o(t)}(\bar{q})$. We emphasize that the last conditional expectation, $\bar{Q}_{A(t,0)}$, has integrated out all variables in $O(t)=(A(t,1:K),L(t,1:K+1))$, and it is therefore only a function of $C_o(t)$. Note that this new representation of the efficient influence curve depends on $\bar{q}$ only through the iteratively defined conditional expectations, $\bar{Q}=(\bar{Q}_{A(t,K+1)},\bar{Q}_{L(t,K+1)},\ldots,\bar{Q}_{A(t,0)})$. Therefore, we can denote the efficient influence curve as $\bar{D}^*_{C_o(t)}(\bar{Q},\bar{g})$ instead. 

\vspace{2mm}
\noindent
The above described representation of the target parameter suggests we can focus our statistical model and TMLE on the conditional expectations instead of the whole conditional density. In particular, our model assumptions on $P_{O(t)\mid C_o(t)}$ can now be replaced by equivalent assumptions on $\bar{Q}_{L(t,j)}$ and $\bar{Q}_{A(t,j)}$. We assume $\bar{Q}_{L(t,j)}=\bar{Q}_{L(j)}$ for  $j=1,\ldots,K+1$ and  $\bar{Q}_{A(t,j)}=\bar{Q}_{A(j)}$ for $j=1,\ldots,K$ are constant in time $t$. Additionally, as before, we assume that for  $j\in {\cal A}_1$, $g_{t,j}$ is known for each $t$ since the probability of assigning treatment is controlled by the experimenter. On the other hand, for $j\in {\cal A}_2$ we assume that $g_{t,j}$ is described by a common (in time $t$) $\bar{g}_j$. For notational simplicity, we denote all $g_{t,j}$ with $\bar{g}_j$. With that, we have redefined our statistical model ${\cal M}^N$ and ${\cal M}(C_o(t))$ for each $C_o(t)$. We reiterate the target parameter, focusing only on the iterative conditional expectation representation:
\[
\bar{\Psi}(\bar{Q})=\frac{1}{N}\sum_{t=1}^N \bar{Q}_{A(0)}(C_o(t)).
\]
The TMLE will now estimate these functions $\bar{Q}_{L(t,j)}(C_o(t),O(t))$ sequentially, with $\bar{Q}_{A(0)}(C_o(t))$ being the last one. 

\vspace{2mm}
\noindent
As before, we construct the initial estimator of $\bar{g}_0$, namely $\bar{g}_N$. The initial estimator $\bar{Q}_{N,L(K+1)}$ of $\bar{Q}_{L(K+1)}$ is obtained based on the following loss function:
\[
\begin{array}{l}
L(\bar{Q}_{L(K+1)})(O^N)\\
=-\sum_{t=1}^N \left\{Y(t)\log \bar{Q}_{L(K+1)}(O(t),C_o(t))+(1-Y(t))\log(1-\bar{Q}_{L(K+1)}(O(t),C_o(t))\right\}
\end{array}
\]
To put it in more context, we note that one could fit a logistic linear regression model of binary outcome $Y(t)$ onto covariates extracted from $L(t,1:K),A(t,1:K)$ and $C_o(t)$. Conditioning on $L(t,1:K),A(t,1:K)$ and $C_o(t)$ essentially treats data records $(Y(t),L(t,1:K),A(t,1:K),C_o(t))$, $t=1,\ldots,N$, as i.i.d. In particular, this estimation strategy would correspond with maximizing the empirical log-likelihood, $-L(\bar{Q}_{L(K+1)})(O^N)$, over a parametric model for $\bar{Q}_{L(K+1)}$. While simple and intuitive, in general we don't expect logistic linear regression to be the appropriate model for the outcome. Instead we advocate for the use of the Super-Learner based on online cross-validated risk, for example. With this estimation strategy, we once again treat $t$-specific data records as i.i.d. in the candidate estimators, while relying on appropriate cross-validation schemes for dependent settings (for instance, one might use online cross-validation). 

\vspace{2mm}
\noindent
Given this initial estimator $\bar{Q}_{N,L(K+1)}$, we proceed to define a parametric working model $\{\bar{Q}_{N,L(K+1),\epsilon}:\epsilon\}$ with finite-dimensional parameter $\epsilon$ so that $\epsilon = 0$ denotes $\bar{Q}_{N,L(K+1)}$. We define a universal least favorable submodel with a logistic fluctuation: 
\[
\mbox{Logit} \bar{Q}_{N,L(K+1),\epsilon}=\mbox{Logit} \bar{Q}_{N,L(K+1)}+\epsilon H_{K+1}(\bar{g}_N),
\]
where $H_{K+1}(\bar{g})$ is the clever covariate for the target parameter. Note that  for each $\epsilon$, we have that:
 \[
 \frac{d}{d\epsilon} L(\bar{Q}_{N,L(K+1),\epsilon})(O^N) = \sum_{t=1}^N D^*_{C_o(t),K+1}(\bar{Q}_{N,L(K+1),\epsilon}, \bar{g}_N).
 \]
 \noindent
Further, we define the maximum likelihood estimator of $\epsilon$ given by:
\[
\epsilon_N^{K+1}=\arg\min_{\epsilon} L(\bar{Q}_{N,L(K+1),\epsilon})(O^N).
\]
Notice that $\epsilon_N^{K+1}$ corresponds with fitting a univariate logistic regression model with covariate $H_{K+1}(\bar{g}_N)$ and offset $\mbox{Logit}\bar{Q}_{N,L(K+1)}$
based on data $(Y(t),H_{K+1}(\bar{g}_N)(C_o(t),O(t)) )$, $t=1,\ldots,N$. The updated TMLE fit of $\bar{Q}_{0,L(K+1)}$ is given by $\bar{Q}_{N,L(K+1)}^*=\bar{Q}^{K+1}_{N,L(K+1),\epsilon_N}$. Due to this TMLE-step, we solve the following estimating equation:
\[
\begin{aligned}
0 &= \frac{1}{N}\sum_{t=1}^N D^*_{C_o(t),K+1}(\bar{Q}_{N,L(K+1)}^*,\bar{g}_N)(O(t)) \\
&= \frac{1}{N}\sum_t H_{K+1}(\bar{g}_N)(C_o(t),O(t))(Y(t)-\bar{Q}_{L(K+1),N}^*).
\end{aligned}
\]
Note that we now have the function $(C_o(t),O(t)) \rightarrow \bar{Q}_{N,L(K+1)}^*(C_o(t),L(t,1:K),A(t,1:K))$. 
Further, we define the following: 
\[
\bar{Q}_{N,A(K)}^*(\cdot)=\int_{a(K)}\bar{Q}_{N,L(K+1)}^*(\cdot,a(K)) \ \bar{g}^*_K(a(K)\mid C_{a,K}^*(t)).
\]
Note that $\bar{Q}_{N,A(K)}^*$ is now a function of $C_o(t),L(t,1:K)$ and $A(t,1:K-1)$. As such, we can evaluate $\bar{Q}_{N,A(K)}^*$ for each $(C_o(t),O(t))$, $t=1,\ldots,N$. Given $\bar{Q}_{N,A(K)}^*$, we define the appropriate loss for $\bar{Q}_{L(K)}$ as: 
\[
\begin{array}{l}
L(\bar{Q}_{L(K)})(O^N)\\
=-\sum_{t=1}^N \left\{\bar{Q}_{N,A(K)}^*\log \bar{Q}_{L(K)}+(1-\bar{Q}_{N,A(K)}^*)\log(1-\bar{Q}_{L(K)}\right\}(C_o(t),O(t)).
\end{array}
\]
For example, one could fit a logistic regression for  outcome $\bar{Q}_{N,A(K)}^*$ onto covariates extracted from $L(t,1:K-1),A(t,1:K-1)$ and $C_o(t)$, as if the data records corresponding to 
$(\bar{Q}_{N,A(K)}^*(O(t),C_o(t))$,  $L(t,1:K-1),A(t,1:K-1),C_o(t))$ are i.i.d. and the outcome is binary. This would correspond with maximizing the empirical log-likelihood, $-L(\bar{Q}_{L(K)})(O^N)$, over a parametric model for $\bar{Q}_{L(K)}$. As discussed previously, in general we advocate for the use of Super-Learning based on the online cross-validated risk for the initial estimation step. 

\vspace{2mm}
\noindent
We define $\bar{Q}_{N,L(K)}$ as the resulting estimator, obtained by using logistic regression or Super-Learning as discussed above. We can now define the logistic fluctuation model for  $\bar{Q}_{N,L(K)}$:
\[
\mbox{Logit}\bar{Q}_{N,L(K),\epsilon}=\mbox{Logit}\bar{Q}_{N,L(K)}+\epsilon H_{K}(\bar{g}_N).
\]
Further, we define the MLE of $\epsilon$ as before:
\[
\epsilon_N^K=\arg\min_{\epsilon} L(\bar{Q}_{N,L(K),\epsilon})(O^N).
\]
Notice that $\epsilon_{N}^K$ corresponds with fitting a univariate logistic regression model with covariate $H_{K}(\bar{g}_N)$ and offset $\mbox{Logit}\bar{Q}_{N,L(K)}$
based on data $(\bar{Q}_{N,A(K)}^*,H_{K}(\bar{g}_N(O(t),C_o(t)) )$ for all $t=1,\ldots,N$. Once again, we note that for all $\epsilon$ we have that:
\[
\frac{d}{d\epsilon} L(\bar{Q}_{N,L(K),\epsilon})(O^N) =\sum_t D^*_{C_o(t),K}(\bar{Q}_{N,L(K),\epsilon}, \bar{g}_N)(O(t)).
\]
The updated TMLE fit or $\bar{Q}_{N,L(K)}$ is given by $\bar{Q}_{N,L(K)}^*=\bar{Q}^K_{N,L(K),\epsilon_N}$. As a consequence of the $K$-specific TMLE targeting step, we have that:
\begin{align*}
0 &= \frac{1}{N} \sum_{t=1}^N D^*_{C_o(t),K}(\bar{Q}_{N,L(K+1)}^*,\bar{Q}_{N,L(K)}^*,\bar{g}_N)(O(t)) \\
&= \frac{1}{N} \sum_{t=1}^N H_{K}(\bar{g}_N)(C_o(t),O(t))(\bar{Q}_{N,A(K)}^*-\bar{Q}_{N,L(K)}^*(C_o(t),O(t)).
\end{align*}
Analogue to the calculation performed for $\bar{Q}_{N,A(K)}^*$, we can now compute $\bar{Q}_{N,A(K-1)}^*$ by integrating out $A(K-1)$ with respect to $\bar{g}^*_{K-1}$. 

\vspace{2mm}
\noindent
We have now showed the general procedure for performing sequential regression based TMLE, concentrating on the first two iterations in the iterative conditional expectation representation of the target parameter. We further iterate the above described process, until we obtain the targeted estimator $\bar{Q}_{N,A(0)}^*$. We note that the targeted estimator  $\bar{Q}_{N,A(0)}^*$ is derived by integrating out $A(0)$ in $\bar{Q}_{N,L(1)}^*$ w.r.t.  the stochastic intervention $\bar{g}^*_0$ of the first intervention node $A(t,0)$.  This yields a general function $C_o\rightarrow \bar{Q}_{N,A(0)}^*(C_o)$, which can be applied to $C_o(t)$  providing $\bar{Q}_{N,A(0)}^*(C_o(t))$, for all $t=1,\ldots,N$.
\vspace{2mm}
\noindent
With that, we formally define the sequential regression based TMLE as:
\[
\bar{\Psi}(\bar{Q}_N^*)=\frac{1}{N}\sum_{t=1}^N\bar{Q}_{N,A(0)}^*(C_o(t))\] of 
\[
\bar{\Psi}(\bar{Q}_0)=\frac{1}{N}\sum_{t=1}^N \bar{Q}_{0,A(0)}(C_o(t))=\frac{1}{N}\sum_{t=1}^N E_0(Y_{\bar{g}^*}(t)\mid C_o(t)).
\]
Most importantly, the targeting steps have enabled for solving estimating equations for each $j=K+1,\ldots,1$. Therefore, we have that for all $j$:
\begin{align*}
0 &= \frac{1}{N} \sum_{t=1}^N D^*_{C_o(t),j}(\bar{Q}_{N,A(j)}^*,\bar{Q}_{N,L(j)}^*,\bar{g}_N)(O(t)) \\
&= \frac{1}{N} \sum_{t=1}^NH_{j}(\bar{g}_N)(C_o(t),O(t))(\bar{Q}_{N,A(j)}^*-\bar{Q}_{N,L(j)}^*(C_o(t),O(t)).
\end{align*}

\vspace{2mm}
\noindent
{\bf Analysis of the sequential regression TMLE:}
We once again apply Theorem \ref{thgen} in order to perform analysis of the sequential regression based TMLE, with the target parameter being the average over time of context-specific causal effects with multiple time-point interventions. We first define the necessary conditions for the next theorem below.
\vspace{2mm}
\noindent
\begin{enumerate}
\item Define ${\cal F}$ to be a class of multivariate, real valued cadlag functions on an Euclidean cube $[0,\tau]$ containing ${\cal C}\times {\cal O}$ with sectional variation norm $\pl f\pl_v^*$, $f \in \cal F$, bounded by an universal constant $M<\infty$. We assume $ \{(C_o,o)\rightarrow D^*_{C_o}(\bar{Q},\bar{g})(o):\bar{Q}\in {\cal Q}, \bar{g}\in {\cal G}\}\subset{\cal F}$. 

\item Assume $\frac{1}{N} \sum_{t=1}^N P_{0,C_o(t)} \{ D^*_{C_o(t)}(\bar{Q}_N^*,\bar{g}_N) - D^*_{C_o(t)}(\bar{Q}^*,\bar{g}_0)\}^2 \rightarrow_p 0$ as $N\rightarrow\infty$ for some possibly misspecified limit $\bar{Q}^*$.

\item Assume negligible or asymptotic linearity of the remainder, such that:
\[
\frac{1}{N} \sum_{t=1}^N R_{2,C_o(t)}(\bar{Q}_N^*,\bar{g}_N,\bar{Q}^*,\bar{g}_0)=o_P(N^{-1/2}),
\] 
and
\[
\frac{1}{N} \sum_{t=1}^N R_{2,C_o(t)}(\bar{Q}^*,\bar{g}_N,\bar{Q}_0,\bar{g}_0)=M_N(f_1)+o_P(N^{-1/2})
\]
for some $f_1$, where $M_N(f)=\frac{1}{N} \sum_{t=1}^N \{f(C_o(t),O(t))-P_{0,C_o(t)}f\}$. Note however, that if $\bar{Q}^*=\bar{Q}_0$, we do not need the second part of the assumption since: $$\frac{1}{N} \sum_{t=1}^N R_{2,C_o(t)}(\bar{Q}^*,\bar{g}_N,\bar{Q}_0,\bar{g}_0) = 0.$$

\item Let $\bar{f}=D^*_{C_0}(\bar{Q}^*,\bar{g}_0)+f_1$. Assume $\frac{1}{N} \sum_{t=1}^N P_{0,C_o(t)}\bar{f}^2\rightarrow\sigma^2_0$.
\end{enumerate}

\begin{theorem} [Sequential regression based TMLE] \label{thmultgint2}
Let $\bar{Q}_N^*$ be the above described sequential regression TMLE of $\bar{Q}_0$, such that: $$\frac{1}{N} \sum_{t=1}^N D^*_{C_o(t)}(\bar{Q}_{N,\epsilon_N},\bar{g}_N)(O(t)) =0.$$ If C1, C2, C3 and C4 hold, then:
\vspace{3mm}
\[
\sqrt{N} (\bar{\Psi}(\bar{Q}_N^*) - \bar{\Psi}(\bar{Q}_0)) \Rightarrow N(0,\sigma^2),
\] 
\vspace{5mm}
\noindent
where  $\sigma^2_0$ is the limit of $\frac{1}{N} \sum_{t=1}^N P_{0,C_o(t)}\bar{f}^2$ and $\bar{f}=D^*_{C_0}(\bar{Q}^*,\bar{g}_0)+f_1$.
\end{theorem}

\vspace{3mm}
\noindent
The proof is an immediate application of Theorem \ref{thgen}.


\section{Adaptive design that learns the optimal individualized treatment rule within a single time-series}\label{sect5}

In this section we develop crucial theoretical foundations for the adaptive learning of the optimal individualized treatment rule based on a single unit. In particular, this section provides important contributions to the field of personalized medicine, whose general focus is on identifying which treatments and preventions will be effective for which individual. A treatment rule for a patient is an individualized treatment strategy based on the history accrued up to the most current time point. A reward is measured on the patient at repetitive units, and optimality is meant in terms of maximization of the mean reward at a particular time $t$. We emphasize the significance of our methodology for adaptive randomized trials within a single unit, which are tailored to approximate an optimal treatment rule as the number of time points grows. Similarly to previous sections, we define the data structure $O(t)=(A(t),Y(t),W(t))$, the $C_o(t)$-specific model and target parameter, and the average across time $t$ of context-specific target parameters. After having defined the estimation problem, we present the TMLE and apply our general Theorem \ref{thgen} to establish its asymptotic normality. 

\subsection{Statistical formulation}

{\bf Data:} Let the observed data be $O(t)=(A(t),Y(t),W(t))$, $t=1,\ldots,N$, where $O(t)$ is of a fixed dimension in time $t$, and is an element of a Euclidean set ${\cal O}$. Let $A(t)\in \{0,1\}$ be a binary treatment and $Y(t)$ the subsequent outcome. We assume that the space $\mathcal{O}$ is bounded, so that without loss of generality, we may assume that the outcome (rewards) are between and bounded away from 0 and 1. Additionally, we denote $W(t)$ as all other information collected after $A(t)$ that will be included in the history $C_o(t+1)$ for the next record $O(t+1)$, with history being defined as in Section 2.1. Finally, let $O^N=(O(t):t=1,\ldots,N)$ and let $P^N$ denote its probability measure. As before, we can factorize the probability density of the data according to the time ordering as follows:
\begin{align*}
p^N(o) &= \prod_{t=1}^{N}p_{a(t)}(a(t) | \bar{o}(t-1))  \prod_{t=1}^{N}p_{y(t)}(y(t) | \bar{o}(t-1),a(t)) \\ 
&\phantom{{}=0} \prod_{t=1}^{N}p_{w(t)}(w(t) | \bar{o}(t-1),y(t),a(t)) .
\end{align*}
Here, $p_{a(t)}$, $p_{y(t)}$ and $p_{w(t)}$ denote the conditional probability densities of $A(t)$, $Y(t)$ and $W(t)$ given the relevant past. We define $\mu_a$, $\mu_y$ and $\mu_w$ as the corresponding dominating measures. 

\vspace{2mm}
\noindent
{\bf Statistical model:} We assume that $P_{O(t)\mid \bar{O}(t-1)}$ depends on $\bar{O}(t-1)$ through a summary measure $C_o(t)=C_o(\bar{O}(t-1))\in {\cal C}$ of fixed dimension. As before, this conditional distribution is denoted with $P_{C_o(t)}$. From the likelihood stated above, we can see that the density $p_{C_o(t)}(a(t),y(t),w(t)\mid C_o(t))$ factorizes into three conditional densities corresponding to $A(t)$, $Y(t)$, and $W(t)$, respectively. If $g_{a(t)}$ is unknown, then we assume that $g_{a(t)} = g_{a(t)}(a(t)\mid C_o(t))$ is stationary in time $t$. We also assume that $q_{y(t)} = q_{y(t)}(y(t)\mid C_y(t))$ is stationary in time, and we make no stationarity assumptions on $q_{w(t)}$. We denote these densities with $g_{a(t)}$, $q_{y(t)}$ and $q_{w(t)}$ as defined in Section \ref{sect3}, with corresponding fixed-dimensional relevant histories $C_a(t)=C_o(t)$, $C_y(t)=(C_o(t),A(t))$ and $C_w(t)=(C_o(t),A(t),Y(t))$. Let $\bar{g}$ be the common conditional density of $A(t)$, given $C_o(t)$, and, if it is known we also denote it with $g_{a(t)}$ or $g_t$. Let $\bar{q}_y$ be the common conditional density of $Y(t)$, given $(A(t),C_o(t))$. Additionally, we define $\bar{Q}(C_o(t),A(t))=E_{P_{C_o(t)}}(Y(t)\mid C_o(t),A(t))$ to be the conditional mean of $Y(t)$ given $C_o(t)$ and $A(t)$. As such, we have that $\bar{Q}(C_o(t),A(t))=\int y \bar{q}_y(y\mid C_o(t),A(t))$, and $\bar{Q}$ is a common function across time $t$. We emphasize that we put no restrictions on $\bar{Q}$, but $\bar{g}$ might be modeled or even known. We suppress dependence of the conditional density on $q_{w(t)}$ in future reference, as this factor plays no role in estimation. In particular, $q_{w(t)}$ does not affect the efficient influence curve of $\Psi_{C_o(t)}$, allowing us to act as if $q_{w(t)}$ is known. We define $\theta=(\bar{g},\bar{q}_y)$ and let $\Theta={\cal G}\times {\cal Q}$ be the cartesian product of the two nonparametric parameter spaces for $\bar{g}$ and $\bar{q}_y$. Let $p_{\theta,C_o(t)}$ and $p_{\theta}^N$ be the density for $O(t)$ given $C_o(t)$ and $O^N$, implied by $\theta=(\bar{g},\bar{q}_y)$. This defines the statistical model ${\cal M}(C_o(t))$ for $P_{C_o(t)}$, and the model ${\cal M}^N$ for the data distribution $P^N$ of $O^N$.

\vspace{2mm}
\noindent
{\bf Estimating optimal treatment rule based on a parametric working model:}
First, we consider estimating the optimal treatment rule based on a parametric working model. Consider a treatment rule $C_o(t)\rightarrow d(C_o(t))\in \{0,1\}$ that maps the history $C_o(t)$ into a treatment decision for $A(t)$. We define a parametric working model for $\bar{q}$ indexed by parameter $\theta$ such that $\{\bar{q}_{\theta}:\theta\}$. Notice that under the specified working model, we have that:
$$\bar{Q}_{\theta}(C_o(t),a) = E(Y(t)\mid C_o(t),A(t)=a)=\int y\bar{q}_{\theta}(y\mid C_o(t),a).$$

\vspace{2mm} 
\noindent
We proceed to define the true conditional treatment effect \[
B_0(C_o(t))\equiv E_0(Y(t)\mid C_o(t),A(t)=1)-E_0(Y(t)\mid C_o(t),A(t)=0),
\]
which can be expressed as $B_{\theta}(C_o(t))=\bar{Q}_{\theta}(C_o(t),1)-\bar{Q}_{\theta}(C_o(t),0)$ under the parametric working model. The optimal treatment rule for $A(t)$ for the purpose of maximizing $Y(t)$ is given by: $$d_0(C_o(t))=I(B_0(C_o(t))>0).$$
Under the parametric working model, we note that the optimal treatment rule can be represented as: $$d_{\theta_0}(C_o(t))=I(B_{\theta_0}(C_o(t))>0).$$

\vspace{2mm} 
\noindent
Define $\theta_{t-1}$ to be the maximum likelihood estimate of $\theta_0$ based on the most current history, $\bar{O}(t)$, and according to the working model $\bar{q}_{\theta}$. We note that we could define $C_o(t)$ such that for each time point $t$, $\theta_{t-1}$ is included in the relevant history $C_o(t)$ for $O(t)$. We now define a current estimate of the rule as:
\[
d(C_o(t))=I(B_{\theta(t-1)}(C_o(t))>0).
\]
We emphasize that if the parametric model is very flexible, $B_{\theta(t-1)}$ might be a good approximation of the true conditional treatment effect $B_0(C_o(t))$. In that case, $d(C_o(t))$ is a good approximation of the optimal rule $d_0(C_o(t))$. Nevertheless, we argue that $\theta_{t-1}$ will converge to $\theta_0$ defined by a Kullback-Leibler projection of the true $\bar{q}_0$ onto the working model $\{\bar{q}_{\theta}:\theta\}$. Consequently, the rule $d(C_o(t))$ will converge to a fixed  $I(B_{\theta_0}(C_o(t))>0)$  as $t$ converges to infinity.

\vspace{2mm} 
\noindent
{\bf Estimating optimal treatment rule with machine learning approaches:}
Instead of considering a parametric working model, we explore estimation of the optimal treatment rule based on more flexible, possibly nonparametric approaches drawn from the machine learning literature. 
As in the previous subsection, we define $B_t(C_o(t))$ to be an estimator of the true blip function, $B_0(C_o(t))$, based on the most recent observations up to time $t$, $\bar{O}(t-1)$. In particular, we consider estimators studied in our previous work, including online Super-Learner of $\bar{Q}_0$ which provides convenient computational and statistical properties for dense time-series data \cite{online2014, onlinecv2018}. Additionally, we might consider Super-Learner that targets $B_0$ directly \cite{luedtke2016super}. Similarly as mentioned in the previous section, we can view $B_t(C_o(t))$ as just another univariate covariate extracted from the past, and include it in our definition of $C_o(t)$.  If $B_t$ is consistent for $B_0$, then the rule $d(C_o(t))$ will converge to the optimal rule $I(B_0(C_o(t))>0)$.

\vspace{2mm} 
\noindent
{\bf Target parameter:}
First, we consider the $C_o(t)$-specific conditional counterfactual mean under the treatment rule $d(C_o(t))$.  At each time $t$, we define the target parameter $\Psi_{C_o(t)}(P_{C_o(t)})$ with $\Psi_{C_o(t)}(P_{C_o(t)}) : {\cal M}(C_o(t))\rightarrow\openr$ as:

\begin{equation}
\Psi_{C_o(t)}(P_{C_o(t)})=E_{P_{C_o(t)}}(Y(t)\mid C_o(t),A(t)=d(C_o(t))).
\end{equation}
We note that $\Psi_{C_o(t)}(P_{C_o(t)}) = \int y \bar{q}_y(y\mid C_o(t),d(C_o(t)))d\mu_y(y)$, representing the conditional mean outcome of $Y(t)$ under the treatment decision $d(C_o(t))$. 
The efficient influence curve for $\Psi_{C_o(t)}(P_{C_o(t)})$ is given by:

\begin{equation}\label{efficadapt}
D^*_{C_o(t)}(\bar{Q})(O(t))=\frac{I(A(t)=d(C_o(t))}{\bar{g}(A(t)\mid C_o(t))}( Y(t)-E_{P_{C_o(t)}}(Y(t)\mid C_o(t),A(t))).
\end{equation}

\vspace{2mm} 
\noindent 
In line with our previous analysis, we describe another interesting target parameter defined as the average of $C_o(t)$-specific counterfactual  means under the treatment rule. In particular, the target parameter on ${\cal  M}^N$, $\Psi^N:{\cal M}^N\rightarrow\openr$ of the data distribution $P^N\in {\cal M}^N$ is defined as:
 \[
\bar{\Psi}(\bar{Q})=\frac{1}{N} \sum_{t=1}^N \Psi_{C_o(t)}(\bar{Q}).
\]
We emphasize that $\bar{\Psi}(\bar{Q})$ is a data dependent target parameter since its value depends on the realized $C_o(t)$, $t=1,\ldots,N$.

\vspace{2mm} 
\noindent 
{\bf Adaptive treatment assignment mechanism:}
In this subsection, we describe the important case when the treatment assignment is controlled by the experimentalist. First, note that a treatment rule could assign $A(t)=d_{\theta_{t-1}}(C_o(t))$ deterministically, therefore assigning treatment decisions according to the best estimate of the optimal treatment rule based on the current history with probability one. A stochastic treatment rule for $\bar{g}(a\mid C_o(t))$ is defined as a random perturbation around $d_{\theta_{t-1}}(C_o(t))$, so that it might assign $d_{\theta_{t-1}}(C_o(t))$ with high probability. 

\vspace{2mm} 
\noindent
Until a a sufficient number of time points is reached necessary to begin to learn the optimal rule, treatment is assigned equiprobably regardless of the past history. Once we have collected enough single-unit time points in order to obtain a primary estimate of the treatment rule, more time points are collected sequentially. Additional time-points collected are exploited in order to learn the optimal treatment rule, which is then approximated by a stochastic of deterministic treatment rule from which the next treatment assignment is drawn conditionally on the next observed relevant history.

 \subsection{Defining the TMLE}
 
Let $L(\bar{Q})(C_o(t),O(t))$ be a loss function for $\bar{Q}$, defined such that we have the following:
 \[
P_{\bar{Q}_0,C_o(t)}L(\bar{Q}_0) = \min_{\bar{Q}} P_{\bar{Q}_0,C_o(t)} L(\bar{Q}).
\]
Therefore, the true $\bar{Q}_0$ minimizes the risk of $L(\bar{Q})$ under $P_{\bar{Q}_0,C_o(t)}$. In particular, we define the loss function for $\bar{Q}$ as follows:
 \[
L(\bar{Q})(C_o(t),O(t))=-\{Y(t)\log \bar{Q}(C_o(t),A(t))+(1-Y(t))\log(1-\bar{Q}(C_o(t),A(t)))\}.
\]

\noindent
Let  $\bar{Q}_N$ be the initial estimator of $\bar{Q}_0$. Similarly, if  $\bar{g}_0$ is not known, we define $\bar{g}_N$ as the initial estimator of $\bar{g}_0$. If it is known, then $\bar{g}_N$ below just denotes the true $\bar{g}_0$. Otherwise, we let $g_{0,t}$ be known but not stationary in time $t$. Given $\bar{Q}_N,\bar{g}_N$, we define a parametric working model $\{\bar{Q}_{N,\epsilon}:\epsilon\}$ with finite-dimensional parameter $\epsilon$ so that $\epsilon = 0$ denotes $\bar{Q}_N$.  In particular, we define a parametric family of fluctuations of the initial estimator with fluctuation parameter $\epsilon$, along with an appropriate loss function defined above, so that the linear combination of the components of the derivate of the loss evaluated at $\epsilon = 0$ span the efficient influence curve at $\theta_N=(\bar{Q}_N,\bar{g}_N)$.  Given the initial estimator $\theta_N$ of $\theta_0$, we compute the maximum likelihood estimator of $\epsilon$, given by:
\[
\epsilon_N=\arg\min_{\epsilon} \frac{1}{N}\sum_{t=1}^N L(\bar{Q}_{N,\epsilon})(C_o(t),O(t)).
\] 
We will use the logistic fluctuation model  $\mbox{Logit} \bar{Q}_{N,\epsilon}=\mbox{Logit}\bar{Q}_N+\epsilon H(\bar{g}_N)$, where the clever covariate is defined as:
\[
H(\bar{g}_N)=\frac{I(A(t)=d(C_o(t))}{\bar{g}_N(A(t)\mid C_o(t))}.
\]
This is an universal least favorable submodel. The TMLE of $\bar{Q}_0$ is given by $\bar{Q}_N^*=\bar{Q}_{N,\epsilon_n}$, and it solves the efficient score equation:
\[
\frac{1}{N} \sum_{t=1}^N D^*_{C_o(t)}(\bar{Q}_N^*,\bar{g}_0)(O(t)) = 0.
\]

\subsection{Analysis of the TMLE}

Recall the definition of our target parameter as the average of $C_o(t)$-causal effects, where 
\[
\bar{\Psi}(\bar{Q})=\frac{1}{N} \sum_{t=1}^N \Psi_{C_o(t)}(\bar{Q}) = \frac{1}{N} \sum_{t=1}^N E_{P_{C_o(t)}}(Y(t) \mid C_o(t),A(t)=d(C_o(t))).
\] 

\noindent
We will consider the case that the treatment mechanism $\bar{g}_0$ is known. Consider the following conditions for the next theorem.
\vspace{2mm}
\noindent
\begin{enumerate}
\item Define ${\cal F}$ to be a class of multivariate, real valued cadlag functions on an Euclidean cube $[0,\tau]$ containing ${\cal C}\times {\cal O}$ with sectional variation norm $\pl f\pl_v^*$, $f \in \cal F$, bounded by an universal constant $M<\infty$. Assume $\{(C_o,o)\rightarrow D^*_{C_o(t)}(\bar{Q},\bar{g}_0)(o):\bar{Q}\in {\cal Q}\}\subset{\cal F}$. 

\item Assume $\frac{1}{N} \sum_{t=1}^N P_{\theta_0,C_o(t)}\{D_{C_o(t)}^*(\bar{Q}_{N}^*,\bar{g}_0)-D_{C_o(t)}^*(\bar{Q}^*,\bar{g}_0)\}^2\rightarrow_p 0$ as $N \rightarrow \infty$ for some possibly misspecified limit $\bar{Q}^*$.

\item Assume $\frac{1}{N}\sum_t {P}_{\theta_0,C_o(t)}D^*_{C_o(t)}(\bar{Q}^*,\bar{g}_0)^2$ converges to a fixed $\sigma^2_0$.

\end{enumerate}

\noindent
\begin{theorem}[Adaptive design learning the optimal treatment rule TMLE] \label{thadaptoit}\ 
Consider the case where $\bar{g}_0$ as known, and let $\bar{Q}_N^*$ be the one-step TMLE so that $$\sum_t D^*(\bar{Q}_N^*,\bar{g}_0)(C_o(t),O(t)) = 0.$$ If the above conditions C1,C2 and C3 hold, then:
\vspace{3mm}
\[
\sqrt{N}(\bar{\Psi}(\bar{Q}_N^*)-\bar{\Psi}(\bar{Q}_0))\Rightarrow N(0,\sigma^2_0),
\] 
\vspace{3mm}
\noindent
where  $\sigma^2_0$ can be consistently estimated with $\sigma^2_N=\frac{1}{N} \sum_{t=1}^N \{ D^*_{C_o(t)}(\bar{Q}^*_N,\bar{g}_0)\}^2$.
\end{theorem}

\vspace{3mm}
\begin{proof}

Recall that
\[
R_{2,C_o(t)}(\bar{Q},\bar{g},\bar{Q}_0,\bar{g}_0)=\frac{\bar{g}-\bar{g}_0}{\bar{g}}(d(C_o(t))\mid C_o(t))(\bar{Q}-\bar{Q}_0)(C_o(t),d(C_o(t)).
\]
Since $\bar{g}_0$ is known, this second order remainder equals zero, and we have the following second order expansion:
\[
\frac{1}{N}\sum_t\Psi_{C_o(t)}(\bar{Q})-\frac{1}{N}\sum_t \Psi_{C_o(t)}(\bar{Q}_0)= \frac{1}{N}\sum_t (D^*_{C_o(t)}(\bar{Q}_N^*,\bar{g}_0) - P_{\theta_0,C_o(t)}) D^*_{C_o(t)}(\bar{Q}_N^*,\bar{g}_0). 
\]
By Lemma \ref{lemmaasympequi}, the martingale process $(M_N(f) : f \in {\cal F})$ defined by
\[
M_N(f)= \frac{1}{N} \sum_{t=1}^N \{f(C_o(t),O(t)) - P_{\theta_0,C_o(t)} f \}
\]
is asymptotically equicontinuous. 
By condition C2, it now follows that $M_N(D^*_{C_o(t)}(\bar{Q}_{N}^*,\bar{g}_0))-M_N(D^*_{C_o(t)}(\bar{Q}^*,\bar{g}_0))=o_P(N^{-1/2})$. Therefore, the second order expansion becomes
\begin{eqnarray*}
\bar{\Psi}(\bar{Q}_N^*) - \bar{\Psi}(\bar{Q}_0) &=&
\frac{1}{N} \sum_{t=1}^N  \left\{ D^*_{C_o(t)}(\bar{Q}^*,\bar{g}_0)(O(t)) - P_{\theta_0,C_o(t)} D^*_{C_o(t)}(\bar{Q}^*,\bar{g}_0) \right\}\\
&&+o_P(N^{-1/2}).
\end{eqnarray*}

\vspace{2mm}
\noindent
By C3, $\frac{1}{N} \sum_{t=1}^N {P}_{\theta_0,C_o(t)} D^*_{C_o(t)}(\bar{Q}^*)^2$ converges to a fixed $\sigma^2_0$, where  $\sigma^2_0$ can be consistently estimated with $\sigma^2_N=\frac{1}{N}\sum_t\{ D^*_{C_o(t)}(\bar{Q}^*_N)\}^2$. By the martingale central limit theorem if follows that
\[
N^{1/2} (\bar{\Psi}(\bar{Q}_N^*) - \bar{\Psi}(\bar{Q}_0)) \Rightarrow N(0,\sigma^2_0),
\]
which proves our result.
\end{proof}

\section{Simulation Study}\label{sect6}

The next subsection, 7.1, describes a simulation study for evaluating the TMLE of average over time of context-specific means for a single time point intervention for several different data-generating distributions. The second subsection (7.2) describes an adaptive trial and a simulation study evaluating the TMLE of the average over time of a rule-specific mean outcome, where this rule represents the best estimate of the optimal rule at that time point. All of the results generated, as well as a software implementation for both the TMLE of average over time of context-specific means for a single time point intervention and adaptive sequential learning the optimal individualized treatment rule parameters are freely available online \cite{malenica2017tstmle}. 

\subsection{Simulation 1: Average over time of context-specific causal effects of single time point intervention}

\noindent
In this section we present results demonstrating the theoretical properties of the methodology presented in Section \ref{sect3}. In particular, we focus on the average over time of $C_o(t)$-specific causal effects of a single time-point intervention on the subsequent outcome. Consider the data structure as defined in Section \ref{sect3}, with $O(t) = (A(t),Y(t),W(t))$ for $t=1,\dots, N$. We explore several different settings that might be of relevance considering actual time-series data. For simplicity we omit missingness and censoring, but note that such settings can be easily incorporated in out treatment variable $A$. In the following simulations we consider binary outcome and treatment, but note that the results will be comparable for continuous outcome. Unless specified otherwise, all results are generated based on 500 Monte Carlo draws used to evaluate the performance of the TMLE estimator of the average over time context-specific causal effect of a single time intervention. We remind that the target parameter of interest is the context-specific average treatment effect, with average being taken over time. In particular, we are interested in the following parameter: 
\[
\Psi^N(P^N) = \frac{1}{N} \sum_{t=1}^N \left\{ \mathbb{E}(Y(t) | A(t) = 1, c_o(t)) - \mathbb{E}(Y(t) | A(t) = 0, c_o(t))\right \} 
\]
with $c_o(t)$ denoting the realized $t$-specific $C_o(t)$. 

\vspace{4mm}
\noindent
\textbf{Simulation 1a (simple dependence)}

\vspace{2mm}
\noindent
We explore a scenario with binary treatment ($A(t) \in \{0,1\}$) and outcome ($Y(t) \in \{0,1\}$) first, with simple dependence extending to Markov order 2. We observe covariates $W_1(t)$, $W_2(t)$ and $W_3(t)$ for each $t=1,\dots, N$, with $W_1(t)$ and $W_3(t)$ drawn from a bernoulli distribution and $W_2(t)$ from a discrete uniform distribution. We note that for this scenario, $W(t) = (W_1(t), W_2(t))$ are drawn independently with respect to the observed past  $\bar{O}(t)$. Further, let the treatment variable $A(t)$ be a function of the past up until $t-2$ and depend on $W_1(t-1),W_2(t-1), Y(t-1), A(t-1)$ and $W_3(t-2)$. The outcome variable $Y(t)$ exhibits dependence of order 2, as a function of $W_1(t-1), W_2(t-1), W_3(t-1), A(t), W_1(t-2)$ and $W_3(t-2)$. For notational convenience, we define $O(1:t)$ as $(O(1), \cdots O(t))$. The exact data-generating distribution used is as follows:

\noindent
\begin{small}
\begin{align*}
A(0:4) &\sim Bern(0.5) \\
Y(0:4) &\sim Bern(0.5) \\
W_1 (0:4) &\sim Bern(0.5) \\
W_2 (0:4) &\sim Unif(1,3) \\
W_3 (0:4) &\sim Bern(0.5) \\
A(4:n) &\sim Bern(expit(0.25*W_1(t-1) - 0.2*W_2(t-1) \\
&\phantom{{}=20} + 0.3*Y(t-1) - 0.2*A(t-1) \\
&\phantom{{}=20} + 0.2*W_3(t-2)))\\
Y(4:t) &\sim Bern(expit(0.3 - 0.8*W_1(t-1) \\
&\phantom{{}=20} + 0.1*W_2(t-1) + 0.2*W_3(t-1) \\
&\phantom{{}=20} + A(t) - 0.5*W_1(t-2) \\
&\phantom{{}=20} + 0.2*W_3(t-2) \\
W_1 (4:n) &\sim Bern(0.5) \\
W_2 (4:n) &\sim Unif(1,3) \\
W_3 (4:n) &\sim Bern(0.5).
\end{align*}
\end{small}

\noindent
The initial estimates $\bar{g}_N,\bar{Q}_N$ were obtained using the online version of the Super-Learner algorithm. In particular, our initial ensemble consisted of multiple algorithms, including simple generalized linear models, penalized regressions and extreme gradient boosting \cite{coyle2018sl3}.  For cross-validation, we relied on the online cross-validation scheme, also known as the recursive scheme in the time-series literature.  We report Wald-type confidence intervals, with the asymptotic  variance estimated as: $$\frac{1}{N}\sum_t \{D^*_{C_o(t)}(\theta_N^*)(O(t))\}^2.$$ In particular, we report the coverage of the resulting asymptotic 95$\%$ confidence intervals to evaluate the performance of the proposed method in Table 1. 

\vspace{4mm}
\noindent
\textbf{Simulation 1b (more elaborate dependence)}

\vspace{2mm}
\noindent
Next, we explore the setting where the single time-series exhibits a more elaborate dependence, while keeping the $n$ as defined in Simulation 1a. Effectively, we are decreasing the sample size and therefore testing the performance of our estimator for different finite sample settings, including the most extreme case of $n=100$. In addition, we consider each part of $O(t)$ to exhibit different levels of dependence, including all the covariates in $W(t)$. For this particular simulation, we treat $A(t)$ as randomized, with a simulation mimicking an observational study considered in Simulation 1c. 

\noindent
\begin{small}
\begin{align*}
A(0:7) &\sim Bern(0.5) \\
Y(0:7) &\sim Bern(0.5) \\
W_1 (0:7) &\sim Bern(0.5) \\
W_2 (0:7) &\sim Normal(0,1) \\
A(7:n) &\sim Bern(0.5) \\
Y(7:n) &\sim Bern(expit(1.5*A(t) - A(t-1) \\
&\phantom{{}=20} + 0.5*Y(t-1) -1.1*W_1(t-1) \\
&\phantom{{}=20} + 0.7*Y(t-3) - A(t-5) + W_1(t-7)) \\
W_1 (7:n) &\sim Bern(expit(0.5*W_1(t-1) - 0.5*Y(t-1) + 0.1*W_2(t-1)) \\
W_2 (7:n) &\sim Normal(0.6*A(t-1) + Y(t-1) -W_1(t-1), sd=1) .
\end{align*}
\end{small}

\vspace{4mm}
\noindent
\textbf{Simulation 1c (Observational study, more elaborate functions and dependence)}

\vspace{2mm}
\noindent
Finally, we consider a typical observational study setup with varying level of dependence and variable interactions. In particular, Simulation 1c considers a setting where each part of the likelihood exhibits some level of dependence, including all of the covariates grouped in $W(t)$. As in Simulation 1a and 1b, we keep $n$ at constant levels $n=(100, 500, 1000)$, and report performance of our estimator for very low effective sample size ($n=100$). We include the highly adaptive lasso (HAL) as part of our Super Learner library, in addition to several glms, penalized regressions and extreme gradient boosting. In addition, we test the double robustness property of our estimator for all sample sizes considered previously, $n=(100, 500, 1000)$. 
The exact data-generating distribution used is as follows:

\noindent
\begin{small}
\begin{align*}
A(0:6) &\sim Bern(0.5) \\
Y(0:6) &\sim Bern(0.5) \\
W_1 (0:6) &\sim Bern(0.5) \\
W_2 (0:6) &\sim Normal(0,1) \\
A(6:n) &\sim Bern(expit(0.7*W_1(t-2) - 0.3*A(t-1) \\
&\phantom{{}=20} + 0.2*sin(W_2(t-2)*A(t-3)) \\
Y(6:n) &\sim Bern(expit(1.5*A(t) - (W_1(t-1)*A(t-2))^2  \\
&\phantom{{}=20} + 0.9* sin(W_2(t-4))*A(t-3)*cos(W_2(t-6)) \\
&\phantom{{}=20} - abs(W_2(t-5)>0))\\
W_1 (6:n) &\sim Bern(expit(0.5*W_1(t-1) - 0.5*Y(t-1) + 0.1*W_2(t-1)) \\
W_2 (6:n) &\sim Normal(0.6*A(t-1) + Y(t-1) - W_1(t-1), sd=1).
\end{align*}
\end{small}

\newpage
\noindent
\begin{table*}[hp!] \centering
\begin{small}
\begin{tabular}{@{}lrrrr@{}}\toprule
& \textbf{n} &\textbf{Bias} & \textbf{Variance} & \textbf{Coverage} \\ \midrule

\textbf{Single time-point intervention (1a)} & 1000 & -2.37e-3 & 9.02e-4 & 94.8\\ \hdashline
& 500 & 2.02e-3 & 1.71e-3 & 96.2\\ \hdashline
& 100 & 5.02e-3 & 1.02e-2 & 92.0\\ \hline

\textbf{Single time-point intervention (1b)} & 1000 & -7.09e-4 & 7.58e-4 & 94.0\\ \hdashline
& 500 & 1.16e-2 & 2.07e-3 & 89.6\\ \hdashline
& 100 & 1.73e-2 & 1.30e-2 & 77.4\\ \hline

\textbf{Single time-point intervention (1c)} & 1000 & 4.79e-3 & 9.45e-4 & 91.2\\ \hdashline
\textbf{} & 500 & 7.52e-3 & 1.92e-3 & 93.8\\ \hdashline
\textbf{}& 100 & 3.71e-3 & 1.25e-2 & 81.8\\  \hline

\bottomrule
\end{tabular}
\end{small}
\caption{Bias, variance and 95$\%$ coverage of the TMLE of the average over time context-specific causal effects with a single time-point intervention for Simulations 1a, 1b and 1c at sample sizes $n=1000$, $n=500$ and $n=100$, over 500 Monte Carlo draws.}
\end{table*}

\noindent
\begin{table*}[hp!] \centering
\begin{small}
\begin{tabular}{@{}lrrrr@{}}\toprule
& \textbf{n} &\textbf{Bias} & \textbf{Variance} & \textbf{Coverage} \\ \midrule

\textbf{Qmgc } & 1000 & 1.43e-2 & 1.26e-3 & 88.4\\ \hdashline
\textbf{Qcgm } & 1000 & 1.42e-2 & 1.25e-3 & 88.4\\ \hline

\textbf{Qmgc } & 500 & 1.29e-2 & 2.63e-3 & 89.2\\ \hdashline
\textbf{Qcgm } & 500 & 1.30e-2 & 2.62e-3 & 89.4\\  \hline

\textbf{Qmgc } & 100 & 3.68e-2 & 1.47e-2 & 84.4\\ \hdashline
\textbf{Qcgm } & 100 & -2.62e-2 & 9.78e-3 & 85.8\\ \hline

\bottomrule
\end{tabular}
\end{small}
\caption{Illustration of the double robustness property of our estimator for Simulation 1c with misspecified (m) and correctly specified (c) models for $\bar{g}_N$ and $\bar{Q}_N$ at sample sizes $n=(1000,500,100)$ over 500 Monte Carlo draws.}
\end{table*}

\subsection{Simulation 2: Adaptive design learning the optimal treatment rule}

\noindent
In this section we present results relevant to methodology described in Section \ref{sect5}, concerning the adaptive learning of the optimal individualized treatment rule. We note that the focus of the simulations presented is estimating the optimal rule with machine learning approaches. We consider the same data structure as in Simulation 1, with $O(t) = (A(t),Y(t),W(t))$ for $t=1,\dots, N$, and omit missingness and censoring. We focus on several different data generating mechanisms, and explore performance of our estimator with different initial sample sizes and consequent sequential updates. As before, we consider binary outcome and treatment, but note that the results will be comparable for continuous bounded outcome. Finally, unless specified otherwise, we present coverage of the mean under the current estimate of the optimal individualized treatment rule at each update based on 500 Monte Carlo draws. For each simulation, we set the reference treatment mechanism to a balanced mechanism assigning treatment with probability $P(A(t) = 1) = 0.5$ for the data draw used to learn the initial estimate of the optimal individualized treatment rule.

\vspace{2mm}
\noindent
For small number of time points, $d_{\bar{Q}_{t-1}}$ might not be a good estimate of $d_{\bar{Q}_0}$. As such, assigning the current conditional probability of treatment based on the fixed dimensional summary measure deterministically based on the estimated rule could be ill-advised. In light of that, we define $\{t_n\}_{t \geq 1}$ and $\{e_n\}_{t \geq 1}$ as user-supplied, non-increasing sequences with $t_1 \leq 0.5$, $t_{\infty} > 0$ and $e_{\infty} > 0$. For every $t \geq 1$, we could have the following function $G_n$ over $[-1,1]$ as defined in \cite{chambaz2017tsoit}:
$$G_n(x) = t_n  \text{I}[x \leq -e_n] + (1-t_n) \text{I}[x \geq e_n] + (-\frac{1/2-t_n}{2 e_n^3}x^3 + \frac{1/2-t_n}{2 e_n/3}x + \frac{1}{2}) \text{I}[ - e_n \leq x \leq e_n]$$ 
$G_n$ is used to derive a stochastic treatment rule from an estimated blip function, as a smooth approximation to $x \rightarrow \text{I}[x \geq 0]$ bounded away from 0 and 1, therefore mimicking the optimal treatment rule as an indicator of the true blip function. In particular, we note that $\{e_n\}_{t \geq 1}$ defines the level of random perturbation around the current estimate $d_{\bar{Q}_{t-1}}$ of the optimal rule. Similarly, choosing $t_1 = \cdots = t_n = 0.5$ would yield a balanced stochastic treatment rule. 

\vspace{4mm}
\noindent
\textbf{Simulation 2a (simple dependence)}

\vspace{2mm}
\noindent
As in Simulation 1a, we explore a simple dependence setting first (Markov order 2) with binary treatment ($A(t) \in \{0,1\}$) and outcome ($Y(t) \in \{0,1\}$). The time-varying covariate $W(t)$ decomposes as $W(t) \equiv (W_1(t), W_2(t))$ with binary $W_1$ and continuous $W_2$. The outcome $Y$ at time $t$ is conditionally drawn given $\{A(t), Y(t-1), W_1(t-1)\}$ from a Bernoulli distribution, with success probability defined as $1.5*A(t) + 0.5 * Y(i-1) - 1.1*W_1(i-1)$. We note as in Section  \ref{sect5} that the conditional mean outcome $\bar{Q}_0(A(t),C_o(t))$  defines the true $C_o(t)$-specific treatment effect and thereby the optimal rule $d_0(C_o(t))$ for assigning treatment $A(t)$. We set the reference treatment mechanism to a balanced treatment mechanism assigning treatment with probability $P(A(t) = 1) = 0.5$. In particular, we generate the initial sample of size $t=1000$ and $t=500$  by first drawing a set of four $O(t)$ samples randomly from binomial and normal distributions in order to have a starting point to initiate time dependence. After these first 4 draws $O(1),O(2),O(3),O(4)$, we draw $A(t)$ from a binomial distribution with success probability 0.5, $Y(t)$ from a Bernoulli distribution with success probability dependent on $\{A(t), A(t-1), Y(t-1), W_2(t-1)\}$, followed by $W_1(t)$ conditional on $\{Y(t-1),W_1(t-1),W_2(t-1)\}$ and $W_2(t)$ conditional on $\{A(t-1), Y(t-1), W_1(t-1)\}$.  After $t=1000$ or $t=500$, we continue to draw $O(t)$ as above, but with $A(t)$ drawn from a stochastic intervention approximating the current estimate $d_{\bar{Q}_{t-1}}$ of the optimal rule $d_{\bar{Q}_0}$. This procedure is repeated until reaching a specified final time point indicating the end of a trial. Our estimator of $\bar{Q}_{0}$, and thereby the optimal rule $d_0$, is based on an online super-learner with an ensemble consisting of multiple algorithms, including simple generalized linear models, penalized regressions and extreme gradient boosting \cite{coyle2018sl3}.  For cross-validation, we relied on the online cross-validation scheme, also known as the recursive scheme in the time-series literature. The sequences $\{t_n\}_{t \geq 1}$ and $\{e_n\}_{t \geq 1}$ are chosen constant, with $t_{\infty} = 10\%$ and $e_{\infty} = 5\%$. The TMLEs are computed at  sample sizes  a multiple of 200, and no more than 1800 (for initial $t=1000$) or 1300 (for initial $t=500$), at which point sampling is stopped.  As in previous subsection, we use the coverage of asymptotic 95$\%$ confidence intervals to evaluate the performance of the TMLE in estimating the average across time $t$ of the $d_{\bar{Q}_{t-1}}$-specific mean outcome. The exact data-generating distribution used is as follows:

\noindent
\begin{small}
\begin{align*}
A(0:4) &\sim Bern(0.5) \\
Y(0:4) &\sim Bern(0.5) \\
W_1 (0:4) &\sim Bern(0.5) \\
W_2 (0:4) &\sim Normal(0,1) \\
A(4:t) &\sim Bern(0.5) \\
Y(4:t) &\sim Bern(expit(1.5*A(i) + 0.5*Y(i-1) - 1.1*W_1(i-1) )) \\
W_1 (4:t) &\sim Bern(expit(0.5*W_1(i-1) - 0.5*Y(i-1) + 0.1*W_2(i-1))) \\
W_2 (4:t) &\sim Normal(0.6*A(i-1) + Y(i-1) - W_1(i-1), sd=1) \\
A(t:1800) &\sim d_{\bar{Q}_{t-1}} \\
Y(t:1800) &\sim Bern(expit(1.5*A(i) + 0.5*Y(i-1) - 1.1*W_1(i-1) )) \\
W_1 (t:1800) &\sim Bern(expit(0.5*W_1(i-1) - 0.5*Y(i-1) + 0.1*W_2(i-1))) \\
W_2 (t:1800) &\sim Normal(0.6*A(i-1) + Y(i-1) - W_1(i-1), sd=1).
\end{align*}
\end{small}

\noindent
\textbf{Simulation 2b (more elaborate dependence)}
In Simulation 2b, we explore the behavior of our estimator in cases of more elaborate dependence. As in Simulation 1a, we only consider binary treatment ($A(t) \in \{0,1\}$) and outcome ($Y(t) \in \{0,1\}$), with binary and continuous time-varying covariates. We set the reference treatment mechanism to a balanced treatment mechanism assigning treatment with probability $P(A(t) = 1) = 0.5$, and generate the initial sample of size $t=(1000, 500)$ by sequentially drawing $W_1(t), W_2(t), A(t), Y(t)$ taking into account the appropriate dependence structure specified by the data-generating mechanism. As before, upon the first $t=1000$ or $t=500$ time-points, we continue to draw $O(t)$ with $A(t)$ drawn from a stochastic intervention approximating the current estimate $d_{\bar{Q}_{t-1}}$ of the optimal rule $d_{\bar{Q}_0}$. The estimator of the optimal rule  $d_{\bar{Q}_0}$ was based on an ensemble of machine learning algorithms and regression-based algorithms, with honest risk estimate achieved by utilizing online cross-validation scheme with validation set size of 30. The sequences $\{t_n\}_{t \geq 1}$ and $\{e_n\}_{t \geq 1}$ were set to $10\%$ and $ 5\%$, respectively. The TMLEs are computed at initial $t=1000$ or $t=500$, and consequently at sample sizes being a multiple of 200, and no more than 1800 (or 1300), at which point sampling is stopped. The exact data-generating distribution used is as follows: 

\noindent
\begin{small}
\begin{align*}
A(0:4) &\sim Bern(0.5) \\
Y(0:4) &\sim Bern(0.5) \\
W_1 (0:4) &\sim Bern(0.5) \\
W_2 (0:4) &\sim Normal(0,1) \\
A(4:t) &\sim Bern(0.5) \\
Y(4:t) &\sim Bern(expit(1.5*A(i) + 0.5*Y(i-3) - 1.1*W_1(i-4) )) \\
W_1 (4:t) &\sim Bern(expit(0.5*W_1(i-1) - 0.5*Y(i-1) + 0.1*W_2(i-2))) \\
W_2 (4:t) &\sim Normal(0.6*A(i-1) + Y(i-1) - W_1(i-2), sd=1) \\
A(t:1800) &\sim d_{\bar{Q}_{t-1}} \\
Y(t:1800) &\sim Bern(expit(1.5*A(i) + 0.5*Y(i-3) - 1.1*W_1(i-4) )) \\
W_1 (t:1800) &\sim Bern(expit(0.5*W_1(i-1) - 0.5*Y(i-1) + 0.1*W_2(i-2))) \\
W_2 (t:1800) &\sim Normal(0.6*A(i-1) + Y(i-1) - W_1(i-2), sd=1).
\end{align*}
\end{small}

\begin{table}[htb] \centering
\begin{small}
\begin{tabular}{@{}lrrrrrr@{}}\toprule
& \textbf{$t$} & $\textbf{Cov}_{t}$ & $\textbf{Cov}_{t_1}$ & $\textbf{Cov}_{t_2}$ & $\textbf{Cov}_{t_3}$  & $\textbf{Cov}_{t_4}$\\ \midrule

\textbf{Adaptive Learning the OIT rule (2a)} & 1000  & 90.00 & 93.20 & 93.80 & 94.80 & 94.60\\  \hdashline
\textbf{Adaptive Learning the OIT rule (2a)} & 500  & 92.60 & 94.00 & 95.20 & 95.40 & 95.80\\  \hline

\textbf{Adaptive Learning the OIT rule (2b)} & 1000  & 92.60 & 92.60 & 93.00 & 93.40 & 93.80 \\  \hdashline
\textbf{Adaptive Learning the OIT rule (2b)} & 500  & 89.60 & 90.20 & 89.60 & 90.20 & 89.40\\  \hline

\bottomrule
\end{tabular}
\end{small}

\caption{Design involves adaptive learning of the optimal individualized treatment rule for a single individual, using the online Super Learner with recursive cross-validation scheme to estimate the optimal treatment rule. The first $t$ time points generates $A(t)$ with probability 0.5. TMLEs are computed at $t = \{500, 1000\}$, $t_1 = t+200$, $t_2 = t+400$, $t_3 = t+600$ and $t_4 = t+800$, with sequential updates being of size 200. The sequences $\{t_n\}_{t \geq 1}$ and $\{e_n\}_{t \geq 1}$ are chosen constant, with $t_{\infty} = 10\%$ and $e_{\infty} = 5\%$.  The table above demonstrates the 95$\%$ coverage for the average across time of the counterfactual mean outcome under the current estimate of the optimal dynamic treatment at that time point, over 500 Monte-Carlo draws for Simulations 2a and 2b with initial sample sizes 1000 and 500.}
\end{table}

\begin{figure}[htb]
\centering
     \includegraphics[height=7.0cm, width=13.0cm]{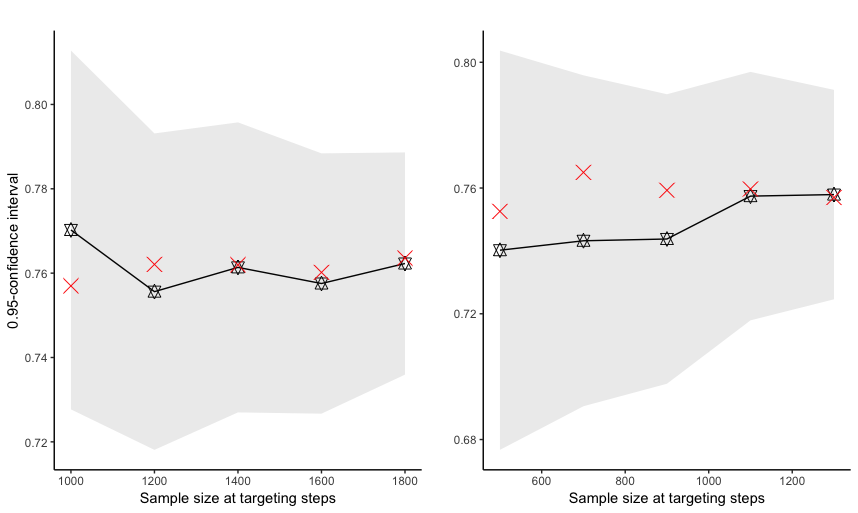}
    \caption{Illustration of the data-adaptive inference of the mean reward under the optimal treatment rule with initial sample size $n=1000$ and $n=500$ for Simulation 2a. The red crosses represent the successive values of the data-adaptive true parameter, with stars representing the estimated parameter with the corresponding $95\%$ confidence interval for the data-adaptive parameter.}
    \label{fig1}
\end{figure}

\begin{figure}[htb]
\centering
        \includegraphics[height=7.0cm, width=13.0cm]{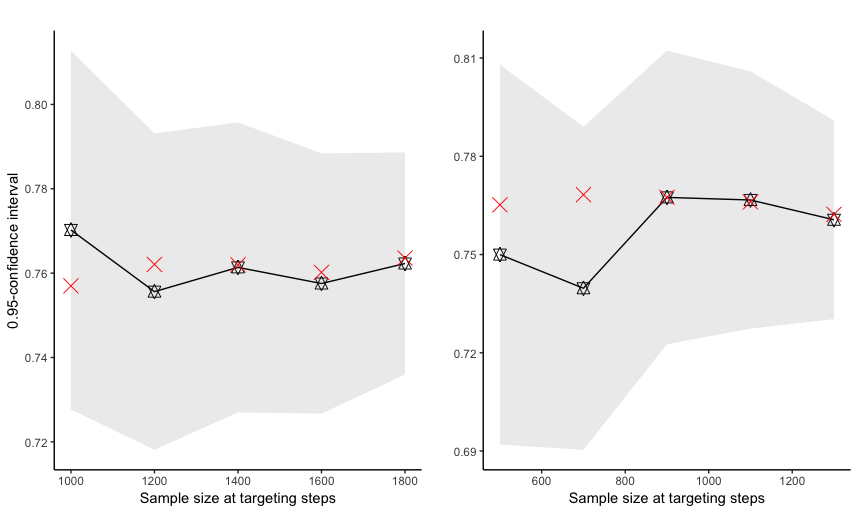}
    \caption{Illustration of the data-adaptive inference of the mean reward under the optimal treatment rule with initial sample size $n=1000$ and $n=500$ for Simulation 2b. The red crosses represent the successive values of the data-adaptive true parameter, with stars representing the estimated parameter with the corresponding $95\%$ confidence interval for the data-adaptive parameter.}
    \label{fig2}
\end{figure}

\section{Discussion}\label{sect7}
In this manuscript, we consider causal inference based on observing a single time series with asymptotic results derived over time $t$. The data setup constitutes a typical longitudinal data structure, where within each $t$-specific time-block one observes treatment and outcome nodes, and possibly time-dependent covariates in-between treatment nodes. Each $t$-specific data record $O(t)$ is viewed as its own experiment in the context of the observed history $C_o(t)$, carrying information about a causal effect of the treatment nodes on the outcome node. A key assumption necessary in order to obtain the presented results is that the relevant history for generating $O(t)$, given the past $\bar{O}(t-1)$, can be summarized by a fixed dimensional summary $C_o(t)$. We note that our conditions allow for $C_o(t)$ to be a function of the whole observed past, allowing us to avoid Markov-order type assumptions that limit dependence on recent past.  Components of $C_o(t)$ that depend on the whole past, such as an estimate of the optimal treatment rule based on $O(1),\ldots,O(t-1)$ will typically converge to a fixed function of a recent past, so that the martingale condition $1/N\sum_t P_{\theta_0,C_o(t)}\{D^*_{C_o(t)}(\theta^*)\}^2\rightarrow \sigma^2_0$ will still hold.

Due to the dimension reduction assumption described in Section 2, each $t$-specific experiment in the sequence of experiments corresponds with drawing from a conditional distribution of $O(t)$, given $C_o(t)$. We assume that this conditional distribution is either constant in time or is parametrized by a constant function. We concentrate on the first setting, as it covers all the applications presented in this manuscript, but note the flexibility of our assumptions. Due to the conditional stationarity assumption, we can asymptotically learn the true mechanism that generates this time-series, even when the model for the mechanism is nonparametric. However, with the exception of parametric models allowing for maximum likelihood estimation, we emphasize that statistical inference for certain target parameters of the data generating mechanism is a challenging problem which requires targeted machine learning. 

In our previous work we provided TMLE for marginal causal parameters, which marginalize over the distribution of $C_o(t)$ \cite{vanderLaan2018onlinets}. For instance, we were interested 
the counterfactual mean of a future (e.g., long term) outcome under a stochastic intervention on a subset of the treatment nodes. This specific parameter addresses the important question regarding the distribution of the outcome at time $t$, had we intervened on some of the past treatment nodes in the time-series. While important, the TMLE of such target parameters are challenging to implement due to their reliance on the density estimation of the marginal density of $C_o(t)$ (averaged across $t$). Additionally, we remark that such marginal causal parameters cannot be robustly estimated if treatment is sequentially randomized, due to lack of double robustness of the second order remainder. 

In this work, we instead focus on context-specific target parameter is order to explore robust statistical inference for causal questions based on observing a single time series on a particular unit. In particular, we note that for each given $C_o(t)$, any intervention-specific mean outcome $EY_{g^*}(t)$ with $g^*$ being a stochastic intervention w.r.t. the conditional distribution of $P_{O(t)\mid C_o(t)}$ represents a well studied statistical estimation problem based on observing $n$ i.i.d. copies. Due to this insight and formulation we are able to repurpose known efficient influence curves and corresponding double robust second order expansions from the i.i.d. literature. Even though we do not have repeated observations from the $C_o(t)$-specific distribution at time $t$, due to the conditional stationarity assumption, the collection $(C_o(t),O(t))$ across all time points represent the analogue of an i.i.d. data set $(C_o(t),O(t))\sim_{iid} P_0$, where $C_o(t)$ can be viewed as a baseline covariate in this typical longitudinal causal inference data structure. Therefore, we estimate the sample-specific counterfactual mean (e.g., sample average treatment effect) $\frac{1}{N}\sum_t E(Y_{g^*}(t)\mid C_o(t))$ using the TMLE of $EY_{g^*}$ developed for  i.i.d. data. We note however that the initial estimation step of the TMLE should still respect the known dependence in construction of the initial estimator, by relying on appropriate estimation techniques developed for dependent data. In particular, we emphasize the importance of time-series based cross-validation schemes (rolling, recursive, fixed and hybrid, to name a few) instead of usual  $V$-fold cross-validation commonly employed for i.i.d settings \cite{Elliott2013}. Similarly, variance estimation can proceed as in the i.i.d case using the relevant i.i.d. efficient influence curve, while ignoring the component corresponding to the baseline covariate $C_o(t)$. This insight relies on the fact that the TMLE in this case  allows for the same linear approximation as the TMLE for i.i.d. data, with the martingale central limit theorem applied to the linear approximation instead. Since the linear expansion of the time-series TMLE for context-specific parameter is an element of the tangent space of the statistical model, our derived TMLE is asymptotically efficient.

To emphasize the importance of our work in applied settings, we provide an exciting application of the context-specific parameter in the settings where the optimal individualized rule is learned adaptively from a single observed time-series. This type of application has important applications in precision medicine, in which one wants to tailor the treatment rule to the individual. In particular, we derive a TMLE which uses only the past data $\bar{O}(t-1)$ of a single unit in order to learn the optimal treatment rule for assigning $A(t)$ to maximize the mean outcome $Y(t)$. Here, we assign the treatment at the next time point $t+1$  according to the current estimate of the optimal rule, allowing for the time-series to learn and apply the optimal treatment rule at the same time. The time-series generated by the described adaptive design within a single unit can be used to estimate, and most importantly provide inference for the average across all time-points $t$ of the counterfactual mean outcome of $Y(t)$ under the estimate $d(C_o(t))$ of the optimal rule at a relevant time point $t$ conditional on $C_o(t)$. Assuming that the estimate of the optimal rule is consistent, as the number of time-points increases, our target parameter converges to the mean outcome one would have obtained had they carried out the optimal rule from the start. As such, we can effectively learn the optimal rule and simultaneously obtain valid inference for its performance. Interestingly, this does not provide inference relative to, for example, the control that always assigns $A(t)=0$. This is due to the fact that by assigning treatment $A(t)$ according to a rule, the positivity assumption needed to learn $\frac{1}{N}\sum_tEY_{A(t)=0}(t)\mid C_o(t))$ is violated. However, we note that one can safely conclude that one will not be worse than this control rule, even when the control rule is equal to the optimal rule. If one is interested in inference for a contrast based on a single time-series, then we advocate for random assignment between the control and estimate of optimal rule. As such, our proposed methodology still allows to learn the desired contrast. 

Finally, we note that while the context-specific parameter enjoys many important statistical and computational advantages as opposed to the marginal target parameter based on a single time-series, the formulation employed in this article is only sensible if one is interested in the causal effect of treatment on a short-term outcome. In particular, if the amount of time necessary to collect outcome $Y(t)$ in $O(t)$ is long, then generating a long time series would take too much time to be practically useful. If one is interested in causal effects on a long term outcome and is willing to forgo utilizing known randomization probabilities for treatment, we advocate for the marginal target parameters as described in our previous work \cite{vanderLaan2018onlinets}.

\newpage
\bibliographystyle{plain}
\bibliography{ddtimeseries}

\end{document}